\documentclass[a4paper]{amsart}
\usepackage{latexsym,bm,stmaryrd}
\usepackage{amsmath,amsthm,amsfonts,amssymb,mathrsfs,pb-diagram}
\usepackage[a4paper,hmargin=26mm,vmargin=28mm]{geometry}
\usepackage{microtype}
\usepackage{mathtools}
\usepackage{mathbbol,wasysym}
\usepackage[all,cmtip]{xy}

\usepackage{pb-diagram}
\usepackage{youngtab}
\usepackage{scalefnt}
\usepackage{tikz}

\newcommand\blam{{\boldsymbol\lambda}}

\newcommand\bmu{{\boldsymbol\mu}}

\let\<=\langle
\let\>=\rangle
\def\({\big(}
\def\){\big)}

\def\Z{\mathbb{Z}}
\def\Q{\mathbb{Q}}
\def\C{\mathbb{C}}
\def\N{\mathbb{N}}

\def\lam{\lambda}
\def\Lam{\Lambda}
\def\Sym{\mathfrak{S}}
\def\eps{\varepsilon}
\def\rf{\mathcal{F}}

\def\bn[#1,#2]{\begin{bmatrix}#1\\#2\end{bmatrix}}

\def\op{{\text{op}}}

\def\lnab{\overline{\nabla}}
\def\ldel{\overline{\Delta}}

\newcommand\bk{\mathbf{k}}

\DeclareMathOperator\lmod{\!-mod}
\DeclareMathOperator\Hom{Hom}
\DeclareMathOperator\End{End}
\DeclareMathOperator\Ext{Ext}
\DeclareMathOperator\soc{soc}
\DeclareMathOperator\im{Image}
\DeclareMathOperator\head{hd}
\DeclareMathOperator\rad{rad}
\DeclareMathOperator\cha{char}

\DeclareMathOperator\id{id}
\DeclareMathOperator\Tr{Tr}
\DeclareMathOperator\Ker{Ker}

\DeclareMathOperator\add{add}

\def\bb{\mathfrak{B}}

\def\co{\mathcal{O}}
\def\fh{\mathfrak{h}^{\ast}}
\def\bk{{\overline{K}}}

\title{Tilting modules, dominant dimensions and Brauer-Schur-Weyl duality}

\author{Jun Hu}
\address[Jun Hu]{School of Mathematics and Statistics, Beijing Institute of Technology, Beijing, 100081, P.R.~China}
\email{junhu404@bit.edu.cn}

\author{Zhankui Xiao}
\address[Zhankui Xiao]{School of Mathematical Sciences, Huaqiao University,
Quanzhou, Fujian, 362021, P. R. China}
\email{zhkxiao@hqu.edu.cn}

\subjclass[2010]{16D90, 20G05, 20G43, 05E10}
\keywords{Quasi-hereditary algebras, tilting modules, dominant dimensions, Brauer algebras, symplectic groups}

\numberwithin{equation}{section}
\newtheorem{prop}[equation]{Proposition}
\newtheorem{thm}[equation]{Theorem}

\newtheorem{cor}[equation]{Corollary}
\newtheorem{conj}[equation]{Conjecture}

\newtheorem{lem}[equation]{Lemma}

\theoremstyle{definition}
\newtheorem{dfn}[equation]{Definition}
\theoremstyle{remark}
\newtheorem{rem}[equation]{Remark}
\newtheorem{exmp}[equation]{Example}

\begin{document}

\begin{abstract}
In this paper we use the dominant dimension with respect to a tilting module to study the double centraliser property. We prove that if $A$ is a quasi-hereditary algebra with a simple preserving duality and $T$ is a faithful tilting $A$-module, then $A$ has the double centralizer property with respect to $T$. This provides a simple and useful criterion which can be applied in many situations in algebraic Lie theory. We affirmatively answer a question of Mazorchuk and Stroppel by proving the existence of a unique minimal basic tilting module $T$ over $A$ for which $A=\End_{\End_A(T)}(T)$.
As an application, we establish a Schur-Weyl duality between the symplectic Schur algebra $S_K^{sy}(m,n)$ and
the Brauer algebra $\mathfrak{B}_n(-2m)$
on the space of dual partially harmonic tensors under certain condition.
\end{abstract}

\maketitle
\setcounter{tocdepth}{1}

\maketitle

\section{Introduction}


Let $k$ be a field. Let $A$ be a finite dimensional $k$-algebra with identity element. Let $A\lmod$ be the category of finite dimensional left $A$-modules. For any $M\in A\lmod$, we use $\add(M)$ to denote the full subcategory of direct summands of finite direct sums of $M$.

Let $T\in A\lmod$. We define $A':=\End_A(T)$. Then $T\in A'\lmod$. We next define $A'':=\End_{A'}(T)$. Then there is a canonical algebra homomorphism $A\rightarrow A''$. Similarly, we define $A''':=\End_{A''}(T)$. It is well-known that the canonical algebra homomorphism $A'\rightarrow A'''$ is an isomorphism.

\begin{dfn} Let $T\in A\lmod$. We say $A$ has the double centraliser property with respect to $T$ if the canonical algebra homomorphism $A\rightarrow A''$ is surjective.
\end{dfn}

\begin{exmp} Let $_AA$ be the left regular $A$-module. Then $A$ has the double centraliser property with respect to $_AA$. In fact, $A'=A^{\rm{op}}$ and $A''=A$.
\end{exmp}

\begin{exmp} If $P\in A\lmod$ is a progenerator, then $A$ has the double centraliser property with respect to $P$.
\end{exmp}

The double centralizer property plays a central role in many part of the representation theory in algebraic Lie theory. For example, the Schur-Weyl duality between
the general linear group $GL(V)$ and the symmetric group $\Sym_r$ on the $r$-tensor space $V^{\otimes r}$ (\cite{Weyl}, \cite{CL}, \cite{DP}) implies that the Schur algebra $S(n,r)$ has the double centralizer property with respect to $V^{\otimes r}$. Similarly, the Schur-Weyl duality between the symplectic group $Sp(V)$ (resp., orthogonal group $O(V)$) and the specialized Brauer algebra $\mathfrak{B}_n(-\dim V)$ (resp., $\mathfrak{B}_n(\dim V)$)
on the $n$-tensor space $V^{\otimes n}$ (\cite{Brauer}, \cite{Brown}, \cite{DP}, \cite{DDH}, \cite{DH}) implies that the symplectic Schur algebra (resp., the orthogonal Schur algebra) has the double centralizer property with respect to $V^{\otimes n}$. For quantized version of these classical Schur-Weyl dualities, we refer the readers to \cite{CP},  \cite{Du}, \cite{DPS}, \cite{Ha}, \cite{Hu2} and \cite{J}. The combinatorial $\mathbb{V}$-functor (due to Soergel \cite{So}) plays a crucial role in the study of the principal blocks of the BGG category $\mathcal{O}$ of any semisimple Lie algebras. The key property of this functor relies on the double centralizer property of the corresponding basic projective-injective module. A similar idea is used in the study of the category $\mathcal{O}$ of the rational Cherednik algebras \cite{GGOR}. For more examples and applications of the double centralizer property in higher Schur-Weyl duality, quantum affine Schur-Weyl duality, etc., we refer the readers to \cite{BK}, \cite{CP2} and \cite{DDF}.

If $T\in A\lmod$  is a faithful $A$-module, then the double centralizer property of $T$ is often closely related to the fully faithfulness of the hom functor $\Hom_A(T,-)$ on projectives. Recall that the hom functor $\Hom_A(T,-)$ is said to be fully faithful on projectives if for any projective modules $P_1,P_2\in A\lmod$, the natural map $$\begin{aligned}
\theta: \Hom_A(P_1,P_2)&\rightarrow\Hom_{(\End_A(T))^{\rm{op}}}(\Hom_A(T,P_1),\Hom_A(T,P_2))\\
f&\mapsto\theta_f: h\mapsto f\circ h .
\end{aligned}$$ is an isomorphism.

For a faithful $A$-module $T$, it is well-known that $A$ has the double centraliser property with respect to $T$ if
and only if the hom functor $\Hom_A(T,-)$ is fully faithful on injectives.
The following result relates the double centralizer property of $T$ to the fully faithfulness of the hom functor
$\Hom_A(T,-)$ on projectives and we leave its proof to the readers.

\begin{lem}\label{ff2} Suppose that $A$ has an anti-involution $\ast$. Let $M$ be a faithful projective submodule of the left regular $A$-module $A$ such that $M=Ae$, where $0\neq e\in A$ is an idempotent with $e^\ast=e$. Then $A$ has the double centralizer property with respect to $M$ if and only if the hom functor $\Hom_A(M,-)$ is fully faithful on projectives.
\end{lem}


Let $T\in A\lmod$  be a faithful $A$-module. When $T$ is not semisimple, it is often difficult to check the double centralizer property of $A$ with respect to $T$ (i.e., whether $A=\End_{\End_A(T)}(T)$ or not) directly.
K\"onig, Slung{\aa}rd and Xi in \cite{KSX} studied the double centralizer property using the notion of dominant dimension. To state their result, we recall the following definition.

\begin{dfn}\text{(\cite[2.5]{KSX})} Let $\mathcal{C}$ be a subcategory of $A\lmod$. Let $M\in A\lmod$ and $C\in\mathcal{C}$. A homomorphism $f: M\rightarrow C$ is called a left $\mathcal{C}$-approximation of $M$ if and only if the induced morphism $
\Hom_A(f,-): \Hom_A(C,D)\rightarrow\Hom_A(M,D)$ is surjective for all objects $D$ in $\mathcal{C}$.
\end{dfn}

\begin{rem}\label{rem1} 1) For any finite dimensional $k$-modules $M, T\in A\lmod$, if $\Hom_{A}(M,T)$ has $k$-dimension $n$, then any homomorphism $M\rightarrow T^{\oplus n}$ is a left $\add(T)$-approximation of $M$.

2) Let $1\leq r\in\N$. Since there is a Morita equivalence between $\End_A(T)$ with $\End_A(T^{\oplus r})$ which sends the $\End_A(T)$-module $T$ to the
$\End_A(T^{\oplus r})$-module $T^{\oplus r}$, it follows that $$
\End_{\End_A(T^{\oplus r})}(T^{\oplus r})\cong\End_{\End_A(T)}(T).
$$
\end{rem}

\begin{dfn}\text{(\cite[2.1]{KSX})}  Let $T, X\in A\lmod$. Then the dominant dimension of $X$ relative to $T$ is the supremum of all $n\in\N$ such that there exists an exact sequence $$
0\rightarrow X\rightarrow T_1\rightarrow T_2\rightarrow\cdots\rightarrow T_n ,
$$ with all $T_i$ in $\add(T)$.
\end{dfn}

The following theorem gives a necessary and sufficient condition for which $A$ has the double centraliser property with respect to a faithful $A$-module $T$.

\begin{thm}\text{(\cite[2.8]{KSX}, \cite[2.1]{AS}, \cite{Ta})}\label{ASthm} Let $T\in A\lmod$. Then the canonical map $A\rightarrow\End_{\End_A(T)}(T)$ is an isomorphism if and only if there exists
an injective left $\add(T)$-approximation $0\rightarrow A\overset{\delta}{\rightarrow}T^{\oplus r}$ which can be continued to an exact sequence
$0\rightarrow A\overset{\delta}{\rightarrow}T^{\oplus r}\overset{\varepsilon}{\rightarrow} T^{\oplus s}$ for some $r,s\in\N$.
\end{thm}

In particular, the above condition means that there exists an injective left $\add(T)$-approximation of $A$ and the $T$-dominant dimension of $A$ is at least two. In general, it is relatively easy to make $\delta$ into an $\add(T)$-approximation, but it is hard to show that the cokernel of the map $\delta$ can be embedded into $T^{\oplus s}$ for some $s\in\N$. By the way, the above theorem actually holds for any finitely generated algebra over a commutative noetherian domain, though we only concentrate on the finite dimensional algebras over a field in this paper.

The starting point of this work is to look for a simple and effective way to verify the above-mentioned embedding property of the cokernel of the map $\delta$.
In many examples of double centralizer property arising in algebraic Lie theory, $T$ is often a tilting module over a finite dimensional quasi-hereditary algebra or even a standardly stratified algebra. The following theorem, which gives a sufficient condition for the double centralizer property with respect to a tilting module over a finite dimensional standardly stratified algebra, is the first main result of this paper.

\begin{thm}\label{mainthm1} Let $A$ be a finite dimensional standardly stratified algebra in the sense of \cite{CPS}. Let $T\in A\lmod$ be a tilting module. Suppose that there is an integer $r\in\N$ such that for any $\lam\in\Lam^{+}$, there is an embedding $\iota_\lam: \Delta(\lambda)\hookrightarrow T^{\oplus r}$ as well as an epimorphism $\pi_\lam: T^{\oplus r}\twoheadrightarrow\overline{\nabla}(\lambda)$ as $A$-modules, then $T$ is a faithful module over $A$ and $A$ has the double centraliser property with respect to $T$. That is, $$
A=\End_{\End_A(T)}\bigl(T\bigr).
$$
\end{thm}

Note that any quasi-hereditary algebra over a field is an example of standardly stratified algebras. Our second and the third main results focus on the finite dimensional quasi-hereditary algebra with a simple preserving duality.
The second main result of this paper gives a simple criterion on $T$ for which $A$ has the double centralizer with respect to $T$.

\begin{thm}\label{mainthm2} Let $A$ be a quasi-hereditary algebra with a simple preserving duality $\circ$. Let $T$ be a faithful tilting module in $A\lmod$. Then $A$ has the double centralizer property with respect to $T$. In particular, the $T$-dominant dimension of $A$ is at least two.
\end{thm}

By \cite[Corollary 2.4]{MS}, there exists a faithful basic tilting module $T\in A\lmod$ such that $A=\End_{\End_A(T)}(T)$. The following theorem is the third main result of this paper, which affirmatively answer a question of  Mazorchuk and Stroppel (see \cite[Remark 2.5]{MS}) on the existence of minimal basic tilting module $T$ for which $A$ has the double centralizer property.

\begin{thm}\label{mainthm3} Let $A$ be a quasi-hereditary algebra with a simple preserving duality. Then there exists a unique faithful basic tilting module $T\in A\lmod$ such that \begin{enumerate}
\item $A=\End_{\End_A(T)}(T)$; and
\item if $T'$ is another faithful tilting module satisfying $A=\End_{\End_A(T')}(T')$, then $T$ must be a direct summand of $T'$.
\end{enumerate}
\end{thm}

The fourth main result of this paper deals with a concrete situation of Brauer-Schur-Weyl duality related to the space of dual partially harmonic tensors. We refer the readers to Section 4 for unexplained notations below.

\begin{thm}\label{mainthm4} Suppose that $\cha K>\min\{n-f+2m,n\}$. Then there is an exact sequence of $S_K^{sy}(m,n)$-module homomorphisms: \begin{equation}\label{fprime}
0\rightarrow  S_{f,K}^{sy}(m,n)\overset{\delta_{f,K}}{\rightarrow} (V_K^{\otimes n}/V_K^{\otimes n}\bb_{n,K}^{(f)})^{\oplus r}\overset{{\eps}_{f,K}}{\rightarrow}(V_K^{\otimes n}/V_K^{\otimes n}\bb_{n,K}^{(f)})^{\oplus s} ,
\end{equation}
such that the map $\delta_{f,K}$ is a left $\add(V_K^{\otimes n}/V_K^{\otimes n}\bb_{n,K}^{(f)})$-approximation of $S_{f,K}^{sy}(m,n)$. In particular, the natural map $S^{sy}_K(m,n)\rightarrow\End_{\bb_{n,K}/\bb_{n,K}^{(f)}}(V_K^{\otimes n}/V_K^{\otimes n}\bb_{n,K}^{(f)})$ is surjective.
\end{thm}

The content of the paper is organised as follows. In Section 2, we first recall the notions of standardly stratified algebras and their basic properties and then give the first main result Theorem \ref{mainthm1} of this paper. In Section 3, we shall focus on the quasi-hereditary algebra with a simple preserving duality. Proposition
\ref{keyprop1} is a key step in the proof of the second main result (Theorem \ref{mainthm2}) of this paper. The proof of Proposition \ref{keyprop1} makes use of a homological result \cite[Corollary 6]{MO} of Mazorchuk and Ovsienko for properly stratified algebras. The proof of the third main result Theorem \ref{mainthm3} is also given in this section. As a remarkable consequence of Theorem \ref{mainthm3}, we obtained in Corollary \ref{remcor} that the existence of a unique minimal faithful basic tilting module $T\in A\lmod$ such that any other faithful tilting module $T'\in A\lmod$ must have $T$ as a direct summand. In Section 4, we use the tool of dominant dimension to study the Schur-Weyl duality between the symplectic Schur algebra $S^{sy}(m,n)$ and $\bb_{n}/\mathfrak{B}_{n}^{(f)}$ on the space $V^{\otimes n}/V^{\otimes n}\mathfrak{B}_{n}^{(f)}$of dual partially harmonic tensors, where $V$ is a $2m$-dimensional symplectic space over $K$, and $\mathfrak{B}_{n}^{(f)}$ is the two-sided ideal of the Brauer algebra $\bb_{n}(-2m)$ generated by $e_1e_3\cdots e_{2f-1}$ with $1\leq f\leq [\frac{n}{2}]$. The aim is to prove  the surjectivity of the natural map from $S^{sy}(m,n)$ to the endomorphism algebra of the space $V^{\otimes n}/V^{\otimes n}\mathfrak{B}_{n}^{(f)}$ as a $\bb_n$-module. The fourth main result Theorem \ref{mainthm4} of this paper proves this surjectivity under the assumption $\cha K>\min\{n-f+2m,n\}$. Another surjection from $\bb_n/\bb_n^{(f)}$ to the endomorphism algebra of the space $V^{\otimes n}/V^{\otimes n}\mathfrak{B}_{n}^{(f)}$ as a $KSp(V)$-module is established in an earlier work \cite{Hu3} by the first author of this paper.


\bigskip
\centerline{Acknowledgements}
\bigskip

The first author was support by the National Natural Science Foundation of China. The second author is supported by the NSF of Fujian Province (Grant No. 2018J01002) and the National NSF of China (Grant No. 11871107). Both authors wish to thank the referee for his/her substantial and insightful comments which significantly improves the final presentation of this article.
\bigskip

\bigskip
\section{Standardly stratified algebras and their tilting modules}

The purpose of this section is to gives a sufficient condition for the double centralizer property with respect to a tilting module over a finite dimensional standardly stratified algebra.

Let $K$ be a field and $A$ be a finite dimensional $K$-algebra with identity element. Let $\{L(\lam)|\lam\in\Lam^{+}\}$ be a complete set of representatives of isomorphic classes of simple modules in $A\lmod$. We always assume that $A$ is split over $K$ in the sense that $\End_A(L(\lam))=K$ for any $\lam\in\Lam^{+}$.
For each $\lam\in\Lam^{+}$, let $P(\lam)\in A\lmod$ be the projective cover of $L(\lam)$ and $I(\lam)\in A\lmod$ the injective hull of $L(\lam)$. For any $M,N\in A\lmod$, we define the trace $\Tr_M(N)$ of $M$ in $N$ as the sum of the images of all $A$-homomorphisms from $M$ to $N$.

Let $A$ be a finite dimensional standardly stratified algebra\footnote{Another slightly different class of standardly stratified algebras was introduced and studied in \cite{AHLU1,ADL} under the same name.} in the sense of \cite{CPS}. That means, there is a partial \textit{preorder} ``$\preceq$" on $\Lam^{+}$, and if set (for any $\lam\in\Lam^{+}$) $$
P^{\succ\lam}:=\oplus_{\mu\succ\lam}P(\mu),\quad P^{\succeq\lam}:=\oplus_{\mu\succeq\lam}P(\mu),\quad \Delta(\lam):=P(\lam)/\Tr_{P^{\succ\lam}}(P(\lam)),
$$
then \begin{enumerate}
\item the kernel of the canonical surjection $P(\lam)\twoheadrightarrow \Delta(\lam)$ has a filtration with subquotients $\Delta(\mu)$, where $\mu\succ\lam$; and
\item the kernel of the canonical surjection $\Delta(\lam)\twoheadrightarrow L(\lam)$ has a filtration with subquotients $L(\mu)$, where $\mu\preceq\lam$.
\end{enumerate}
We call $\Delta(\lam)$ the {\textbf{standard module}} corresponding to $\lam$. Note that $\Delta(\lam)$ is the maximal quotient of $P(\lam)$ such that $[\Delta(\lam):L(\mu)]=0$ for all $\mu\succ\lam$. In particular, $\head\Delta(\lam)\cong L(\lam)$. We define the \textbf{proper standard module} $\ldel(\lam)$ to be $$
\ldel(\lam):=P(\lam)/\Tr_{P^{\succeq\lam}}(\rad P(\lam)),
$$
which is the maximal quotient of $P(\lam)$ satisfying $[\rad\ldel(\lam):L(\mu)]=0$ for all $\mu\succeq\lam$. It is clear that there is a natural surjection
$\beta_\lam: \Delta(\lam)\twoheadrightarrow\ldel(\lam)$.

Similarly,  let $I^{\succ\lam}:=\oplus_{\mu\succ\lam}I(\mu)$, $I^{\succeq\lam}:=\oplus_{\mu\succeq\lam}I(\mu)$, we define the {\textbf{proper costandard module}} $\lnab(\lam)$ to be the preimage of $$
\bigcap_{f: I(\lam)/\soc I(\lam)\rightarrow I^{\succeq\lam}}\Ker f
$$
under the canonical epimorphism $I(\lam)\twoheadrightarrow I(\lam)/\soc I(\lam)$. Then $\lnab(\lam)$ is the maximal submodule of $I(\lam)$ satisfying
$[\lnab(\lam)/\soc\lnab(\lam):L(\mu)]=0$ for all $\mu\succeq\lam$. Note that $\soc\lnab(\lam)=\soc I(\lam)\cong L(\lam)$. We define the \textbf{costandard module} $\nabla(\lam)$ to be $$
\nabla(\lam):=\bigcap_{f: I(\lam)\rightarrow I^{\succ\lam}}\Ker f,
$$
which is the maximal submodule of $I(\lam)$ such that $[\nabla(\lam):L(\mu)]=0$ for all $\mu\succ\lam$. In particular, there is a natural embedding $\alpha_\lam: \lnab(\lam)\hookrightarrow\nabla(\lam)$.

We use $\rf(\Delta)$ (resp., $\rf(\lnab)$) to denote the full subcategory of $A\lmod$ given by all $A$-modules having a filtration with all subquotients of the filtration being isomorphic to $\Delta(\lam)$ (resp., $\lnab(\lam)$) for some $\lam\in\Lam^{+}$.
In \cite{F}, Frisk developed the theory of tilting module for standardly stratified algebra. Recall that by a tilting module we mean an object in $\rf(\Delta)\cap\rf(\lnab)$. Let $\{T(\lam)|\lam\in\Lam^+\}$ be a complete set of pairwise non-isomorphic indecomposable tilting modules in $A\lmod$.

\begin{lem}\text{(\cite[Lemma 6, Theorems 9,13]{F})}\label{homolog1} Let $\lam,\mu\in\Lam^{+}$ and $i\in\N$. Let $M\in A\lmod$. Then \begin{enumerate}
\item $M\in\rf(\Delta)$ if and only if $\Ext_A^1(M,\lnab(\nu))=0$ for any $\nu\in\Lam^{+}$;
\item $M\in\rf(\lnab)$ if and only if $\Ext_A^1(\Delta(\nu),M)=0$ for any $\nu\in\Lam^{+}$;
\item we have $$
\Ext_A^i(\Delta(\lam),\lnab(\mu))=\delta_{i0}\delta_{\lam,\mu}K ;
$$
\item $A\in\rf(\Delta)$.
\end{enumerate}
\end{lem}

\begin{proof} (1), (2) and (3) all follow from \cite[Lemma 6, Theorems 9,13]{F}, while (4) follows from (1) because $A$ is projective and hence $\Ext_A^1(A,N)=0$ for any $N\in A\lmod$.
\end{proof}

\begin{lem}\label{embed} Let $0\rightarrow M_1\overset{\iota}{\rightarrow} M\overset{\pi}{\rightarrow} M_2\rightarrow 0$ be a short exact sequence in $A\lmod$. Let $T\in A\lmod$ such that the induced natural map $\iota_*: \Hom_{A}(M,T)\rightarrow\Hom_A(M_1,T)$ is surjective and there are embeddings $\iota_1: M_1\hookrightarrow T^{\oplus a},\,\,\iota_2: M_2\hookrightarrow T^{\oplus b}$ as $A$-modules for some $a,b\in\N$. Then there is an embedding $\iota_0: M\hookrightarrow T^{\oplus a+b}$ as $A$-modules.
\end{lem}

\begin{proof} By assumption, it is clear that the induced natural map $\iota_{a,*}: \Hom_{A}(M,T^{\oplus a})\rightarrow\Hom_A(M_1,T^{\oplus a})$ is surjective too. Hence there exists $\hat{\iota}_1\in\Hom_A(M,T^{\oplus a})$ such that $\iota_{a,*}(\hat{\iota}_1)=\iota_1$. In other words, \begin{equation}\label{lift1}
\hat{\iota}_1(\iota(y))=\iota_1(y),\,\,\,\forall\,y\in M_1 .
\end{equation}
We now define a map $\iota_0: M\rightarrow T^{\oplus a+b}=T^{\oplus a}\oplus T^{\oplus b}$ as follows: $$
x\mapsto \bigl(\hat{\iota}_1(x),\iota_2(\pi(x))\bigr),\quad\,\forall\,x\in M .
$$
It is easy to check that $\iota_0$ is an injective homomorphism in $A\lmod$. This completes the proof of the lemma.
\end{proof}

\begin{lem}\label{surj} Let $i: N_1\hookrightarrow N_2$ be an embedding in $A\lmod$. Let $0\rightarrow M'\overset{\iota}{\rightarrow} M\overset{\pi}{\rightarrow} M''\rightarrow 0$ be a short exact sequence in $A\lmod$. Suppose that the natural map
$\Hom_{A}(N_2,M)\rightarrow\Hom_A(N_2,M'')$ and the following natural maps $$
\iota'_1: \Hom_{A}(N_2,M')\rightarrow\Hom_A(N_1,M')\quad\,\,\iota'_2: \Hom_{A}(N_2,M'')\rightarrow\Hom_A(N_1,M''),
$$
are all surjective. Then the natural map $\iota'_0: \Hom_{A}(N_2,M)\rightarrow\Hom_A(N_1,M)$ is surjective too.
\end{lem}

\begin{proof} This follows from diagram chasing.
\end{proof}

Now we can give the proof of the first main result of this paper.

\medskip
\noindent
{\textbf{Proof of Theorem \ref{mainthm1}}:} Since $A$ is a projective left $A$-module, $A\in\rf(\Delta)$ by definition. Applying Lemmata \ref{homolog1} and  \ref{embed}, we can get an integer $a\in\N$ and an embedding $\delta_0: A\hookrightarrow T^{\oplus a}$ as $A$-modules. In particular, $T^{\oplus a}$ and hence $T$ is a faithful $A$-module.

Let $\lam\in\Lam^{+}$. We fix a $K$-basis $\{v_1,\cdots,v_{m_\lam}\}$ of $\lnab(\lam)$. Let $\pi_\lam: T^{\oplus r}\twoheadrightarrow\lnab(\lam)$ be the given surjection. For each $1\leq j\leq m_\lam$, we fix an element $u_j\in T^{\oplus r}$ such that $\pi_\lam(u_j)=v_j$, and we denote by $\delta_j^\lam$ the following left $A$-module homomorphism:  $$
\delta_j^\lam: A\rightarrow T^{\oplus r},\quad\, x\mapsto xu_j,\,\,\forall\,x\in A .
$$

We set $r':=a+\sum_{\lam\in\Lam^{+}}rm_\lam$. We now define a map $\delta: A\rightarrow T^{\oplus r'}$ as follows: $$\begin{aligned}
\delta:\,\, A&\rightarrow T^{\oplus r'}=T^{\oplus a}\oplus\bigoplus_{\lam\in\Lam^{+}}(T^{\oplus r})^{\oplus m_\lam}\\
x &\mapsto \delta_0(x)\oplus\bigoplus_{\lam\in\Lam^{+}}(\delta_1^\lam(x)\oplus\cdots\oplus\delta_{m_\lam}^\lam(x)),\,\,\,\forall\,x\in A .
\end{aligned}
$$

Since $\delta_0$ is injective, it is clear that $\delta$ is injective too. Furthermore, by construction, it is easy to see that each basis element $v_j\in\lnab(\lam)$ has a preimage in $\Hom_A(T^{\oplus r'},\lnab(\lam))$ under the natural homomorphism
$\Hom_A(T^{\oplus r'},\lnab(\lam))\rightarrow\Hom_A(A,\lnab(\lam))\cong\lnab(\lam)$ induced by $\delta$. In other words, the natural homomorphism
$\Hom_A(T^{\oplus r'},\lnab(\lam))\rightarrow\Hom_A(A,\lnab(\lam))\cong\lnab(\lam)$ induced by $\delta$ is always surjective. Applying Lemmata \ref{homolog1} and \ref{surj} we can deduce that for any $M\in\mathcal{F}(\lnab)$ the natural homomorphism
$\Hom_A(T^{\oplus r'},M)\rightarrow\Hom_A(A,M)$ induced by $\delta$ is always surjective. In particular, this implies that the injection $0\rightarrow A\rightarrow T^{\oplus r'}$ is a left $\add(T)$-approximation.

The surjection of $\Hom_A(T^{\oplus r'},\lnab(\lam))\rightarrow\Hom_A(A,\lnab(\lam))$ for any $\lam\in\Lam^{+}$ and the fact that $$
\Ext_A^1(T^{\oplus r'},\lnab(\lam))=0$$
imply that $\Ext_A^1(T^{\oplus r'}/A,\lnab(\lam))=0$ and hence $T^{\oplus r'}/A\in\mathcal{F}(\Delta)$ by Lemma \ref{homolog1}. Now applying Lemma \ref{embed} and using the assumption that for any $\mu\in\Lam^{+}$, $\Delta(\mu)\hookrightarrow T^{\oplus k}$ for some $k\in\N$, we can deduce that there is an integer $s'\in\N$ such that $T^{\oplus r'}/A\hookrightarrow T^{\oplus s'}$. In other words, the injective left $\add(T)$-approximation $0\rightarrow A\overset{\delta}{\rightarrow}T^{\oplus r'}$ can be continued to an exact sequence
$0\rightarrow A\overset{\delta}{\rightarrow}T^{\oplus r'}\overset{\varepsilon}{\rightarrow} T^{\oplus s'}$. Finally, using Theorem \ref{ASthm}, we prove the theorem.

\begin{cor}\label{cor to thm1.9}
Let $A$ be a finite dimensional standardly stratified algebra such that the injective hull of
$\Delta(\lambda)$ is projective and the projective cover of $\lnab(\lam)$ is injective for every $\lambda\in \Lam^{+}$.
If $T$ is a projective-injective generator, then $A$ has the double centraliser property with respect to $T$.
\end{cor}

\bigskip
\bigskip
\section{Quasi-hereditary algebra with a simple preserving duality}

In this section we shall focus on the finite dimensional quasi-hereditary algebras over a field with a simple preserving duality. We shall give the proof of the second and third main
results (Theorem \ref{mainthm2}, Theorem \ref{mainthm3}) of this paper for this class of algebras.

\begin{dfn} We say that $A$ has a simple preserving duality if there exists an exact, involutive and contravariant equivalence $\circ: A\lmod\rightarrow A\lmod$ which preserves the isomorphism classes of simple modules.
\end{dfn}

Let $A$ be a finite dimensional standardly stratified algebra with a simple preserving duality. Then for each $\lam\in\Lam^+$, we have $$
\Delta(\lam)^\circ\cong\nabla(\lam),\quad\ldel(\lam)^\circ\cong\lnab(\lam),\quad P(\lam)^\circ\cong I(\lam) .
$$
In particular, $A^\circ$ is an injective left $A$-module.

\begin{dfn}\label{proper}\text{(\cite{AHLU1, ADL, D1, M1})} Let $A$ be a finite dimensional standardly stratified algebra. If $\preceq$ is a partial order on $\Lam^{+}$ and the following conditions are satisfied for all $\lam\in\Lam^{+}$: \begin{enumerate}
\item the kernel of the canonical epimorphism $\ldel(\lam)\twoheadrightarrow L(\lam)$ has a filtration with subquotients
$L(\mu), \mu\prec\lam$;
\item $\Delta(\lam)$ has a filtration with subquotients $\ldel(\lam)$,
\end{enumerate}
then we shall call $A$ a properly stratified algebra.
\end{dfn}

If $\Delta(\lam)=\ldel(\lam)$ for each $\lam\in\Lam^+$, then the properly stratified algebra $A$ is a quasi-hereditary algebra (\cite{CPS0}, \cite{DR1}). In that case, we also have $\lnab(\lam)=\nabla(\lam)$ for any $\lam\in\Lam^{+}$.

Let $M\in A\lmod$. We define $\dim_{\rf(\Delta)}M$ to be the minimal integer $j$ such that there is an exact sequence of the form $$
0\rightarrow M_j\rightarrow M_{j-1}\rightarrow\cdots\rightarrow  M_1\rightarrow M_0\rightarrow M\rightarrow 0,
$$
where $M_i\in\rf(\Delta)$ for any $0\leq i\leq j$; while if no such integer $j$ exists then we define $\dim_{\rf(\Delta)}M:=\infty$.

\begin{lem}\text{(\cite[Corollary 6]{MO})}\label{MO} Assume that $A$ is a properly stratified algebra having a simple preserving duality,
and such that every tilting $A$-module is cotilting.
Let $M\in A\lmod$ with $\dim_{\rf(\Delta)}M=t<\infty$. Then $\Ext_A^{2t}(M,M^\circ)\neq 0$.
\end{lem}

\begin{lem}\label{Delta-flag}
Let $A$ be a quasi-hereditary algebra with a simple preserving duality $\circ$. Let
$0\rightarrow P\rightarrow T\rightarrow K\rightarrow 0$ be a short exact sequence in $A\lmod$ with
$P$ a projective $A$-module and $T$ a tilting $A$-module. Then we have $K\in\rf(\Delta)$.
\end{lem}

\begin{proof}
Applying the duality functor $\circ$, we get another short exact sequence
\begin{equation}\label{one more}
0\rightarrow K^\circ\hookrightarrow T^\circ\twoheadrightarrow P^\circ\rightarrow 0 .
\end{equation}
Suppose that $K\notin\rf(\Delta)$. Then we must have that $\dim_{\rf(\Delta)}K=1$. Applying Lemma \ref{MO}, we get that $\Ext_A^2(K, K^\circ)\neq 0$.

We have the following long exact sequence of homomorphisms:
$$
\cdots\rightarrow\Ext_A^1(P,T)\rightarrow\Ext_A^2(K,T)\rightarrow\Ext_A^2(T,T)\rightarrow\cdots .
$$
Applying Lemma \ref{homolog1} and noting that $P$ is projective, we can deduce that $\Ext_A^1(P,T)=0=\Ext_A^2(T,T)$. It follows that $\Ext_A^2(K,T)=0$.
Since $A$ is quasi-hereditary algebra with a simple preserving duality, we have $T(\lam)^\circ\cong T(\lam)$ for each $\lam\in\Lam^{+}$. It follows that $T^\circ\cong T$. Hence $\Ext_A^2(K,T^\circ)=0$.

On the other hand, from (\ref{one more}) we can get another long exact sequence of homomorphisms:
$$
\cdots\rightarrow\Ext_A^1(K,P^\circ)\rightarrow\Ext_A^2(K,K^\circ)\rightarrow\Ext_A^2(K,T^\circ)\rightarrow\cdots .
$$
Since $P^\circ$ is injective, we have $\Ext_A^1(K,P^\circ)=0$. By the last paragraph, $\Ext_A^2(K,T^\circ)=0$. It follows that $\Ext_A^2(K,K^\circ)=0$, which is a contradiction. This proves that $K\in\rf(\Delta)$.
\end{proof}

The following proposition plays a crucial role in the proof of the second and the third main results of this paper.

\begin{prop}\label{keyprop1} Let $A$ be a quasi-hereditary algebra with a simple preserving duality $\circ$. Let $T$ be a tilting module in $A\lmod$. Suppose there is an embedding $\iota: A\hookrightarrow T$ in $A\lmod$. Then we have that $T/A\in\rf(\Delta)$ and, for any $\lam\in\Lam^{+}$, there exists a surjective homomorphism $T^{\oplus m_\lam}\twoheadrightarrow\nabla(\lam)$ as well as an injective homomorphism
$\Delta(\lam)\hookrightarrow T^{\oplus m_\lam}$, where $m_\lam:=\dim\nabla(\lam)$.
\end{prop}

\begin{proof}
Applying Lemma \ref{Delta-flag} to the short exact sequence $0\rightarrow A\overset{\iota}{\hookrightarrow} T\twoheadrightarrow T/A\rightarrow 0$,
we get $T/A\in\rf(\Delta)$.

Applying Lemma \ref{homolog1}, we get that $\Ext_A^1(T/A,\nabla(\lam))=0$ for any $\lam\in\Lam^{+}$. Thus we have an exact sequence of homomorphisms: $$
0\rightarrow\Hom_A(T/A,\nabla(\lam))\rightarrow\Hom_A(T,\nabla(\lam))\overset{\iota_\ast}{\rightarrow}\Hom_A(A,\nabla(\lam))=\nabla(\lam)\rightarrow\Ext_A^1(T/A,\nabla(\lam))=0 .
$$
It follows that the canonical map $\iota_\ast: \Hom_A(T,\nabla(\lam))\rightarrow\Hom_A(A,\nabla(\lam))=\nabla(\lam)$ induced from $\iota$ is surjective.

Let $\lam\in\Lam^{+}$ and $a:=\iota(1)$. We fix a $K$-basis $\{v_1,\cdots,v_{m_\lam}\}$ of $\nabla(\lam)$. For each $1\leq j\leq m_\lam$, we can choose a map $\rho_j\in\Hom_A(T,\nabla(\lam))$, such that $v_j=\iota_\ast(\rho_j)=\rho_j(a)$. Now we define a map $\rho: T^{\oplus m_\lam}\rightarrow\nabla(\lam)$ as follows: $$
(w_1,\cdots ,w_{m_\lam})\mapsto \sum_{j=1}^{m_\lam}\rho_j(w_j),\,\,\forall\,(w_1,\cdots ,w_{m_\lam})\in T^{\oplus m_\lam}.
$$
It is clear that $\rho$ is a left $A$-module homomorphism and $v_j\in\im(\rho)$ for each $1\leq j\leq m_\lam$. This implies that $\rho$ is a surjective $A$-module homomorphism. By taking duality, we also get an injective homomorphism $\Delta(\lam)\hookrightarrow T^{\oplus m_\lam}$. This completes the proof of the proposition.
\end{proof}

Now we can give the proof of our second main result of this paper.

\medskip
\noindent
{\textbf{Proof of Theorem \ref{mainthm2}}:}  By assumption, $T$ is a faithful module over $A$. Then there exists a natural number $r$ and an embedding $\delta: A\hookrightarrow T^{\oplus r}$ as $A$-modules. By Remark \ref{rem1}, $\delta$ is an injective left $\add(T)$-approximation of $A$. Applying Proposition \ref{keyprop1}, we can find a natural number $m:=\max\{m_\lam|\lam\in\Lam^{+}\}$,  and a surjective homomorphism $T^{\oplus rm}\twoheadrightarrow\nabla(\lam)$. Taking duality, we get an embedding
$\Delta(\lam)\hookrightarrow (T^{\oplus rm})^\circ\cong (T^{\circ})^{\oplus rm}\cong T^{\oplus rm}$. Thus the theorem follows from Theorem \ref{mainthm1}.
\medskip


Using Theorem \ref{mainthm2}, we can easily recover many known double centralizer properties or simplify the proof of the corresponding Schur-Weyl dualities in non-semisimple or even integral situation.

\begin{exmp} Let $n,r\in\N$ and $V$ be an $n$-dimensional vector space over an arbitrary field $K$. Let $q\in K^\times$. Let $\mathscr{H}_q(\Sym_n)$ be the Iwahori-Hecke algebra associated to the symmetric group with Hecke parameter $q$. There is a natural right action of $\mathscr{H}_q(\Sym_n)$ on $V^{\otimes n}$ (\cite{DJ1}). Let $S_q(n,r):=\End_{\mathscr{H}_q(\Sym_n)}(V^{\otimes r})$ be the Dipper-James $q$-Schur algebra over $K$ (\cite{DJ1}). Then $V^{\otimes r}$ is a faithful tilting module over $S_q(n,r)$. As a cellular algebra, $S_q(n,r)$ has an anti-involution $\ast$ which sends its semistandard basis $\varphi_{ST}$ to $\varphi_{TS}$ (cf. \cite[Proposition 4.13]{Mat}). It follows from Theorem \ref{mainthm2} that $S_q(n,r)$ has the double centralizer property with respect to $V^{\otimes r}$. In particular, the $V^{\otimes r}$-dominant dimension of $S_q(n,r)$ is at least two (\cite[Theorem 3.6]{KSX}). Specializing $q$ to $1$, we also get the double centralizer property of the classical Schur algebra $S(n,r)$.
\end{exmp}

\begin{exmp} Let $n,\ell\in\N$ and $\xi,Q_1,\cdots,Q_\ell\in K^\times$. Let $\mathscr{H}_{n}=\mathscr{H}_{n}(\xi, Q_1,\cdots Q_r)$ be the non-degenerate cyclotomic Hecke algebra of type $G(\ell,1,n)$ (\cite{AK}, \cite{DJM}). By definition, $\mathscr{H}_{n}$ is a unital $K$-algebra generated by $T_0,T_1,\cdots,T_{n-1}$ which satisfies the following relations: $$\begin{aligned}
&(T_0-Q_1)\cdots (T_0-Q_\ell)=0, \quad T_0T_1T_0T_1=T_1T_0T_1T_0,\\
&(T_i+1)(T_i-\xi)=0,\quad \forall\,1\leq i<n,\\
&T_jT_k=T_kT_j,\quad \forall\,0\leq j<k-1\leq n-2,\\
&T_rT_{r+1}T_r=T_{r+1}T_rT_{r+1},\quad \forall\,1\leq r<n-1 .
\end{aligned}
$$
For each $1\leq m\leq n$, define $L_m:=T_{m-1}\cdots T_1T_0T_1\cdots T_{m-1}$. For each multicomposition $\blam=(\lam^{(1)},\cdots,\lam^{(\ell)})$ of $n$ with $\ell$-components, we set $a_1=0, a_2=|\lam^{(1)}|,a_3=|\lam^{(1)}|+|\lam^{(2)}|,\cdots,a_n=\sum_{j=1}^{n-1}|\lam^{(j)}|$,
and define $$
m_{\blam}:=\Bigl(\prod_{k=1}^{\ell}\prod_{m=1}^{a_k}(L_m-Q_k)\Bigr)\sum_{w\in\Sym_\blam}T_w,
$$
where $\Sym_\blam\cong\Sym_{\lam^{(1)}}\times\Sym_{\lam^{(2)}}\times\cdots\times\Sym_{\lam^{(\ell)}}$ is the standard Young subgroup of $\Sym_n$ corresponding to $\blam$.

Let $\Lam$ be a subset of the set  $\mathscr{C}_n$ of all multicompositions of $n$ which have $\ell$ components such that if $\blam\in\Lam$ and $\bmu$ is a multipartition of $n$ such that $\bmu\rhd\blam$, then $\bmu\in\Lam$, where $\rhd$ is the dominance partial order on $\mathscr{C}_n$ defined in \cite{Mat}. Let $\Lam^+$ be the set of multipartitions in $\Lam$. We define the cyclotomic tensor space $$
M:=\bigoplus_{\blam\in\Lam}m_\blam\mathscr{H}_n ,
$$
and define the cyclotomic $q$-Schur algebra $$
\mathscr{S}_n:=\End_{\mathscr{H}_n}(M) .
$$
By \cite{DJM}, we know that $\mathscr{S}_n$ is a cellular and quasi-hereditary algebra. By \cite[Lemma 4.8]{LR}, $M$ is a faithful tilting module over $\mathscr{S}_n$. Applying Theorem \ref{mainthm2}, we get that $\mathscr{S}_n$ has the double centralizer property with respect to $M$. That is, $$
\End_{\End_{\mathscr{S}_n}(M)}(M)=\mathscr{S}_n .
$$
This recovers the result \cite[Theorem 4.10]{LR}.
\end{exmp}

\begin{exmp} Let $m,n\in\N$ and $q\in K^\times$. Let $U_v(\mathfrak{sp}_{2m})$ be Drinfeld-Jimbo's quantized enveloping algebra of the symplectic Lie algebra $\mathfrak{sp}_{2m}(\C)$ over the rational functional field $\Q(v)$. Let $U_q(\mathfrak{sp}_{2m}):=K\otimes_{\Z[v,v^{-1}]}\mathbb{U}_v(\mathfrak{sp}_{2m})$, where $\mathbb{U}_v(\mathfrak{sp}_{2m})$ is Lusztig's $\Z[v,v^{-1}]$-form of $U_v(\mathfrak{sp}_{2m})$, $K$ is regarded as a $\Z[v,v^{-1}]$-algebra by specializing $v$ to $q$. Let $V\cong L(\varepsilon_1)$ be the $2m$-dimensional natural representation of  $U_q(\mathfrak{sp}_{2m})$. Let $\mathfrak{B}_{n,q}=\mathfrak{B}_{n}(-q^{2m+1},q)$ be the specialized Birman-Murakami-Wenzl algebra over $K$. We refer the readers to \cite{HuXiao} for its precise definition. There is a natural right action of  $\mathfrak{B}_{n,q}$ on $V^{\otimes n}$ which commutes with the natural left action of $U_q(\mathfrak{sp}_{2m})$. Let $S_q^{sy}(m,n):=\End_{\mathfrak{B}_{n,q}}(V^{\otimes n})$ be associated the symplectic $q$-Schur algebra. By \cite{Oe2}, we know that $S_q^{sy}(m,n)$ is cellular and quasi-hereditary. Moreover, $V^{\otimes n}$ is a faithful tilting module over $S_q^{sy}(m,n)$. Applying Theorem \ref{mainthm2}, we get that ${S}_q^{sy}(m,n)$ has the double centralizer property with respect to $V^{\otimes n}$. That is, $$
\End_{\End_{{S}_q^{sy}(m,n)}(V^{\otimes n})}(V^{\otimes n})={S}_q^{sy}(m,n) .
$$
In particular, the $V^{\otimes n}$-dominant dimension of $S_q^{sy}(m,n)$ is at least two. Specializing $q$ to $1$, we also get the double centralizer property of the classical symplectic Schur algebra $S^{sy}(m,n)$ (\cite{Oe}).
\end{exmp}

The following example is an easier case of the double centralizer property which was already well known before (cf. \cite[Section 2.3]{KSX}). 

\begin{exmp} Let $\mathfrak{g}$ be an arbitrary complex semisimple Lie algebra with root system $\Phi$ and the set $\Pi$ of simple roots. Let $\mathfrak{g}=\mathfrak{n}^-\oplus\mathfrak{h}\oplus\mathfrak{n}^+$ be a triangular decomposition of $\mathfrak{g}$ and $W$ the Weyl group of $\mathfrak{g}$. Let $\mathcal{O}$ be the Bernstein-Gelfand-Gelfand (BGG) category which consists of finitely generated $U(\mathfrak{g})$-modules which are $\mathfrak{h}$-semisimple and locally $U(\mathfrak{n})$-finite (\cite{Hum}). For any $w\in W, \lam\in\mathfrak{h}^*$, we define $w\cdot\lam=w(\lam+\rho)-\rho$, where $\rho$ is the half sum of all positive roots. Let $I\subset\Pi$ be a subset of $\Pi$ and $\mathfrak{p}:=\mathfrak{p}_I$ the associated standard parabolic subalgebra of $\mathfrak{g}$. Let $\mathfrak{p}_I=\mathfrak{l}_I\oplus \mathfrak{u}_I$ be the Levi decomposition of $\mathfrak{p}$. Let $\mathcal{O}^\mathfrak{p}$ denote the corresponding parabolic BGG category \cite{Ro} which is a full subcategory of $\co$ consisting of modules which are $U(\mathfrak{l}_I)$-semisimple and locally $U(\mathfrak{u}_I)$-finite. For each $\lambda\in\fh$, let $\co_\lam$ be the Serre subcategory of $\co$ which is generated by $L(x\cdot\lam)$ for all $x\in W$, let $\mathcal{O}_\lambda^\mathfrak{p}:=\mathcal{O}_\lambda\cap\mathcal{O}^\mathfrak{p}$. The simple module $L(\lambda)$ lies in $\mathcal{O}^\mathfrak{p}$ if and only if $\lambda\in \Lam_{I}^+:=\{\lambda\in\fh|\<\lambda,\alpha^\vee\>\in\mathbb{Z}^{\geq 0},\,\forall\,\alpha\in I\}$. For $\lambda\in \Lam_{I}^+$, let $M_I(\lambda)$ denote the parabolic Verma module with highest weight $\lambda$ and  $P_I(\lambda)$ denote the projective cover of $L(\lambda)$ in $\mathcal{O}^\mathfrak{p}$. Suppose that $\lam\in\Lam_I^+$ is an integral weight. Let $$
P_{I,\lam}:=\bigoplus_{w\in W, w\cdot\lam\in\Lam_I^{+}}P_I(w\cdot\lam),
$$
be a progenerator of $\co_\lam^\mathfrak{p}$. Set $A_\lam^{\mathfrak{p}}:=\End_{\co}(P_{I,\lam})$. Then $\co_\lam^\mathfrak{p}\sim (A_\lam^{\mathfrak{p}})^{\rm{op}}\lmod$. It is well-known that $A_\lam^{\mathfrak{p}}$ is a quasi-hereditary basic algebra equipped with an anti-involution $\ast$ (see \cite[Chapter 1]{Hum}). In particular, $A_\lam^{\mathfrak{p}}\lmod$ has a simple preserving duality and each indecomposable projective over $A_\lam^{\mathfrak{p}}$ can be generated by an $\ast$-fixed primitive idempotent.

A weight $w\cdot\lam\in\Lam_I^+$ is called socular if $L(w\cdot\lam)$ occurs in the socle of some parabolic Verma module $M_I(x\cdot\lam)$. By a result of Irving, $P_I(w\cdot\lam)$ is injective if and only if $w\cdot\lam$ is socular. We define $$
Q_{I,\lam}:=\bigoplus_{\substack{w\in W\\ \text{$w\cdot\lam$ is socular}}}P_I(w\cdot\lam),
$$
Since each $P_I(w\cdot\lam)$ has a standard filtration, each simple module in the socle of  $P_I(w\cdot\lam)$ is labelled by a socular weight. It follows that the injective hull of each  $P_I(w\cdot\lam)$ is a projective-injective module. Thus the basic algebra $A_\lam^{\mathfrak{p}}$ can be embedded into a direct sum of some copies of the basic projective-injective (hence tilting) module $\Hom_{\co}(P_{I,\lam},Q_{I,\lam})$. In particular, $\Hom_{\co}(P_{I,\lam},Q_{I,\lam})$ is a faithful tilting module over $A_\lam^{\mathfrak{p}}$.
Applying Corollary \ref{cor to thm1.9}, we get that $A_\lam^{\mathfrak{p}}$ has the double centralizer property with respect to $\Hom_{\co}(P_{I,\lam},Q_{I,\lam})$. In particular, by Lemma \ref{ff2} the hom functor $\Hom_{\co}(Q_{I,\lam},-)$ is fully faithful on projectives. This recovers earlier results of \cite[Struktursatz 9]{So} and \cite[Theorem 10.1]{Stro1}.
\end{exmp}

Using Theorem \ref{mainthm2}, it is possible to simplify the proof of Schur-Weyl dualities in many non-semisimple situations. General speaking, we have two algebras $A, B$ and an $(A,B)$-bimodule $M$. By a Schur-Weyl duality between $A$ and $B$ on the bimodule $M$ we mean that the following two canonical maps: \begin{equation}\label{ab}
\varphi: A\rightarrow \End_B(M),\quad \psi: B\rightarrow\End_A(M) ,
\end{equation}
are both surjective.

Suppose that there is a Schur-Weyl duality between $A$ and $B$ on the bimodule $M$. That says, both $\varphi$ and $\psi$ are surjective. Then it is obvious that $A$ has the double centralizer property with respect to $M$ and $B$ has the double centralizer property with respect to $M$.

Conversely, if we can show that the image of $\varphi$ in $\End_B(M)$ is a quasi-hereditary algebra with a simple preserving duality, then the surjectivity of
$\varphi$ will follow from the surjectivity of $\psi$ and applying Theorem \ref{mainthm2}. This is because in that case we have $$
A\twoheadrightarrow\im(\varphi)=\End_{\End_{\im(\varphi)}(M)}(M)=\End_{\End_A(M)}(M)=\End_{B}(M) .
$$
This is indeed the case as in many examples of Schur-Weyl dualities, where $A$ often has a highest weight theory with $M$ being a tilting module over $A$, and $B$ is a diagrammatic algebra (symmetric or cellular). It is usually easier to handle the endomorphism algebra $\End_A(M)$ than to handle the endomorphism algebra $\End_B(M)$.

\begin{exmp} Let $G$ be a classical group over an algebraically closed field $K$ with the natural module $V$. Following \cite{G} and \cite[\S2.2]{Do4}, we define $S_r(G):=A_r(G)^\ast$, where $A_r(G)$ is the coefficient space of $V^{\otimes r}$ (which is a coalgebra) in the coordinate algebra $K[G]$. The $K$-algebra $S_r(G)$ is isomorphic to the image of $KG$ in $\End_K(V^{\otimes r})$ and hence acts faithfully on $V^{\otimes r}$. When the set of dominant weights in $V^{\otimes r}$ is saturated in the sense of \cite[A3]{Do3}, then $S_r(G)$ is a generalised Schur algebra. In particular it is quasi-hereditary. If furthermore $V^{\otimes r}$ is a faithful tilting module over $S_r(G)$ then we can apply Theorem \ref{mainthm2}.

In the type $A$ case, let $G=GL(V)$, the general linear group on $V$, and $n:=\dim V$. The Schur-Weyl duality between $KGL(V)$ and the symmetric group algebra $K\Sym_r$ on $V^{\otimes r}$ means that we have the following two surjective homomorphisms: $$
\varphi: KGL(V)\twoheadrightarrow\End_{K\Sym_r}(V^{\otimes r}),\quad \psi: K\Sym_r\twoheadrightarrow\End_{KGL(V)}(V^{\otimes r}) .
$$
In this case, the set of dominant weights in $V^{\otimes r}$ is $$
\Lam^+(n,r)=\{\lam=\lam_1\varepsilon_1+\cdots+\lam_n\varepsilon_n|\sum_{i=1}^n\lam_i=r,\lam_1\geq\lam_2\geq\cdots\geq\lam_n\geq 0, \lam_i\in\Z,\forall\,i\},
$$
which is saturated, $V^{\otimes r}$ is a faithful tilting module over the image of $\varphi$.

In the type $C$ case, let $V$ be a $2m$-dimensional symplectic space, $G=Sp(V)$, the symplectic group on $V$. The Schur-Weyl duality between $KSp(V)$ and the specialized Brauer $\bb_{n,K}(-2m)$ on $V^{\otimes n}$ means that we have the following two surjective homomorphisms: $$
\varphi: KSp(V)\twoheadrightarrow\End_{\bb_{n,K}(-2m)}(V^{\otimes n}),\quad \psi: \bb_{n,K}(-2m)\twoheadrightarrow\End_{KSp(V)}(V^{\otimes n}) .
$$
In this case, the set of dominant weights in $V^{\otimes n}$ is $$
\{\lam=\lam_1\varepsilon_1+\cdots+\lam_m\varepsilon_m|\sum_{i=1}^m\lam_i=n-2f,0\leq f\leq [n/2],\lam_1\geq\lam_2\geq\cdots\geq\lam_m\geq 0, \lam_i\in\Z,\forall\,i\},
$$
which is saturated, $V^{\otimes n}$ is a faithful tilting module over the image of $\varphi$.

In the type $D$ case, we assume $\cha K\neq 2$ and $m\geq 2$. Let $V$ be a $2m$-dimensional orthogonal space, $G=SO(V)$, the special orthogonal group on $V$. In this case, Donkin (\cite[\S2.5]{Do4}) has shown that the set of dominant weights in $V^{\otimes n}$ is again saturated, and $V^{\otimes n}$ is a faithful tilting module over the image of $KSO(V)$ in $\End_K(V^{\otimes n})$. In particular, we see the image has the double centralizer property with respect to $V^{\otimes n}$ by Theorem \ref{mainthm2}. That is, we have the following natural surjective homomorphism: $$
\varphi: KSO(V)\twoheadrightarrow\End_{\End_{KSO(V)}(V^{\otimes n})}(V^{\otimes n}) .
$$
Note that the image of $KO(V)$ in $\End_K(V^{\otimes n})$ is not necessarily equal to the image of $KSO(V)$, and $\End_{KSO(V)}(V^{\otimes n})$ in general does not coincide with $\End_{KO(V)}(V^{\otimes n})$ in this case.

In the type $B$ case, we assume $\cha K\neq 2$ and $m\geq 2$. Let $V$ be a $2m+1$-dimensional orthogonal space, $G=SO(V)$, the special orthogonal group on $V$. In this case, the set of dominant weights in $V^{\otimes n}$ is in general not saturated. Donkin (\cite[\S2.5]{Do4}) has given a sufficient condition in \cite[Page 108,(H)]{Do4} under which the image of $KSO(V)$ in $\End_K(V^{\otimes n})$ is quasi-hereditary. However, in this case, $O(V)$ is generated by $SO(V)$ and an involution $\theta$, and $\theta$ acts as $-\id$ on $V^{\otimes n}$. So the image of $KO(V)$ in $\End_K(V^{\otimes n})$ coincides with the image of $KSO(V)$. Moreover, $\End_{KSO(V)}(V^{\otimes n})=\End_{KO(V)}(V^{\otimes n})$. Thus by \cite[Theorem 1.2]{DH} we can still get the following natural surjective homomorphism: $$
\varphi: KSO(V)\twoheadrightarrow\End_{\End_{KSO(V)}(V^{\otimes n})}(V^{\otimes n}) .
$$
\end{exmp}

\begin{rem} Let $\mathfrak{g}$ be a complex semisimple Lie algebra. Let $U(\mathfrak{g})$ be the universal enveloping algebra of $\mathfrak{g}$ over $\Q$.
Let ${U}_\Z(\mathfrak{g})$ be the Kostant $\Z$-form of $U(\mathfrak{g})$. For any field $K$, we define $U_K(\mathfrak{g}):=K\otimes_{\Z}{U}_\Z(\mathfrak{g})$. The discussion in the above example should also work if we replace $KG$ by $U_K(\mathfrak{g})$, see \cite[\S3]{Do4}. Furthermore, we remark that the argument of \cite[\S3]{Do4} should also work if we replace $U_K(\mathfrak{g})$ with the Lusztig's $\Z[v,v^{-1}]$-form of the Drinfeld-Jimbo quantized enveloping algebra of $\mathfrak{g}$. In that case, the idea of using  Theorem \ref{mainthm2} should be able to simplify the lengthy argument in \cite{Hu2} and to provide a proof of the quantized integral Schur-Weyl dualities in the orthogonal cases as well. Details will be appeared elsewhere.
\end{rem}



Let $\{T(\lam)|\lam\in\Lam^+\}$ be a complete set of pairwise non-isomorphic indecomposable tilting modules over $A$. Recall that the tilting module $\oplus_{\lam\in\Lam^+}T(\lam)$ is called the characteristic tilting module over $A$. By a well-known result of Ringel, we know that $A$ has the double centraliser property with respect to the characteristic tilting module. Note that the characteristic tilting module is a basic tilting module in the sense of the following definition.

\begin{dfn} Let $T\in A\lmod$ be a tilting module. If $T$ is a direct sum of some pairwise non-isomorphic indecomposable tilting modules, then we say that $T$ is a basic tilting module. In general, if $T=\oplus_{\lam\in\Lam^+}T(\lam)^{\oplus r_\lam}$, where $r_\lam\in\N$ for each $\lam$, then we define $$
T_{\rm{basic}}:=\oplus_{\substack{\lam\in\Lam^+\\ r_\lam>0}}T(\lam) .
$$
\end{dfn}

In \cite[Remark 2.5]{MS}, Mazorchuk and Stroppel proposed a question about whether there exists a minimal basic tilting module with respect to which one has the double centraliser property. In the rest of this section we shall give the proof of Theorem \ref{mainthm3}, which affirmatively answers this question.

\medskip
\noindent
{\textbf{Proof of Theorem \ref{mainthm3}}:}  Let $\widetilde{T}:=\oplus_{\lam\in\Lam^+}T(\lam)$ be the characteristic tilting module over $A$. Then $R(A):=\End_A(\widetilde{T})$ is the Ringel dual of $A$, which is again a quasi-hereditary algebra over $K$. We consider the contravariant Ringel dual functor ${\rm R}:=\Hom_A(-,\widetilde{T}): A\lmod\rightarrow\End_A(\widetilde{T})\lmod$ (\cite{R}): $$
M\mapsto {\rm R}(M)=\Hom_A(M,\widetilde{T}),\quad\forall\,M\in A\lmod .
$$
By \cite[Proposition 2.1]{MS}, $\rm{R}$ maps tilting modules to projective modules, and maps projective modules to tilting modules and $\rm{R}$ defines an equivalence of subcategories ${\rm R}: \rf(\Delta^A)\cong\rf(\Delta^{{R}(A)})$, where the $\Delta^A$ means the standard objects in $A\lmod$ and $\Delta^{{R}(A)}$ means the standard objects in ${R}(A)\lmod$.

Let $\widetilde{P}$ be the projective cover of $\widetilde{T}$ in ${R}(A)\lmod$. Then ${\rm R}^{-1}(\widetilde{P})$ is a tilting module in $A\lmod$.
We define $T:=\bigl({\rm R}^{-1}(\widetilde{P})\bigr)_{\rm{basic}}$. Let $r'\in\N$ such that ${\rm R}^{-1}(\widetilde{P})$ is a direct summand of $T^{\oplus r'}$.
Let $\phi: \widetilde{P}^{\oplus r}\twoheadrightarrow\widetilde{T}$ be a surjective homomorphism in ${R}(A)\lmod$. Since $\widetilde{P}^{\oplus r'}, \widetilde{T}\in\rf(\Delta^{{R}(A)})$, it follows that $\Ker\phi\in\rf(\Delta^{{R}(A)})$. Thus $$
0\longrightarrow\Ker\phi\overset{\iota}{\longrightarrow} \widetilde{P}^{\oplus r'}\overset{\pi}{\longrightarrow}\widetilde{T}\longrightarrow 0
$$
is an exact sequence in $\rf(\Delta^{{R}(A)})$. Applying the inverse of the Ringel dual functor $\rm{R}$, we get that $$
0\longrightarrow A\overset{{\rm R}^{-1}(\pi)}{\longrightarrow} ({\rm R}^{-1}(\widetilde{P}))^{\oplus r'}\overset{{\rm R}^{-1}(\iota)}{\longrightarrow}{\rm R}^{-1}(\Ker\phi)\longrightarrow 0 ,
$$
is an exact sequence in $\rf(\Delta^A)$.

By definition, $T:=\bigl({\rm R}^{-1}(\widetilde{P})\bigr)_{\rm{basic}}$. Thus there exist $N\in A\lmod$ and $r'\leq r\in\N$, such that $$
({\rm R}^{-1}(\widetilde{P}))^{\oplus r'}\oplus N=T^{\oplus r}
$$
We use $\iota_0: A\hookrightarrow T^{\oplus r}$ to denote the composition of ${\rm R}^{-1}(\pi)$ with the natural injection $({\rm R}^{-1}(\widetilde{P}))^{\oplus r'}\hookrightarrow T^{\oplus r}$.
Thus $T^{\oplus r}$ and hence $T$ must be faithful tilting modules over $A$. Now applying Theorem \ref{mainthm2} and Remark \ref{rem1}, we can deduce that $A=\End_{\End_A(T)}(T)$.

Now suppose that $T'$ is another faithful tilting module satisfying $A=\End_{\End_A(T')}(T')$. Let $s\in\N$ such that $A\hookrightarrow (T')^{\oplus s}$.
Applying Proposition \ref{keyprop1}, we can get an exact sequence in $\rf(\Delta^A)$: $$
0\rightarrow A\hookrightarrow (T')^{\oplus s}\twoheadrightarrow  (T')^{\oplus s}/A\rightarrow 0 .
$$
Applying the Ringel dual functor $\rm{R}$, we get that $$
0\rightarrow{\rm R}((T')^{\oplus s}/A)\hookrightarrow {\rm R}(T')^{\oplus s}\overset{\text{$h$}}{\twoheadrightarrow}\widetilde{T} \rightarrow 0,
$$
is an exact sequence in $\rf(\Delta^{{R}(A)})$.

Note that ${\rm R}(T')$ is a projective module in ${\rm R}(A)\lmod$. The surjectivity of $h$ in the above exact sequence implies that ${\rm R}(T')^{\oplus s}$ must contains $\widetilde{P}$ as its direct summand. That says, ${\rm R}(T')^{\oplus s}\cong\widetilde{P}\oplus\widetilde{P}'$. Applying the inverse of the Ringel dual functor $\rm{R}$, we get that $(T')^{\oplus s}\cong{\rm R}^{-1}(\widetilde{P})\oplus{\rm R}^{-1}(\widetilde{P}')$. By construction, $T:=\bigl({\rm R}^{-1}(\widetilde{P})\bigr)_{\rm{basic}}$. This implies that $T$ must be isomorphic to a direct summand of $T'$. This completes the proof of the theorem.
\medskip

\begin{cor}\label{remcor} Let $A$ be a quasi-hereditary algebra with a simple preserving duality. Then there exists a unique minimal faithful basic tilting module $T\in A\lmod$ such that any other faithful tilting module $T'\in A\lmod$ must have $T$ as a direct summand.
\end{cor}

\begin{rem}\label{add-rem}
In fact, the same argument can be used to show that Theorems \ref{mainthm2} and \ref{mainthm3} are true for properly stratified algebra $A$ which has a simple preserving duality and that every tilting $A$-module is cotilting.
\end{rem}

\bigskip\bigskip
\section{Brauer-Schur-Weyl duality for dual partially harmonic spaces}\label{xxsec2}

In this section, we shall apply the results in last Section to the study of Brauer-Schur-Weyl duality for dual partially harmonic spaces.


The notion of Brauer algebra was first introduced in \cite{Brauer} when Richard Brauer studied the decomposition of symplectic tensor spaces and orthogonal tensor spaces into direct sums of irreducible modules. Since then there have been a lot of study on the structure and representation of Brauer algebras, see \cite{CVM1, CVM2, DoranWH,
HK1, HK2, HP, Martin, RW, Rui, Wenzl} and references therein.
In this section we only concern about these Brauer algebras with special parameters which play a role in the setting of Brauer-Schur-Weyl duality of type $C$.
Let $m,n\in\Z^{\geq 1}$. The Brauer algebra $\mathfrak{B}_{n,\Z}=\bb_{n}(-2m)_{\Z}$ with parameter $-2m$ over $\Z$ is a unital associative $\Z$-algebra with generators
$s_1,\ldots,s_{n-1},e_1,\ldots,e_{n-1}$ and relations (see \cite{En}):
$$
s_i^2=1,\hspace{12pt} e_i^2=(-2m)e_i,\hspace{12pt}
e_is_i=s_ie_i=e_i,\hspace{12pt} \forall\ 1\leq i\leq n-1,
$$
$$
s_is_j=s_js_i,\hspace{8pt} s_ie_j=e_js_i,\hspace{8pt}
e_ie_j=e_je_i,\hspace{8pt} \forall\ 1\leq i<j-1\leq n-2,
$$
$$
s_is_{i+1}s_i=s_{i+1}s_is_{i+1},\hspace{6pt}
e_ie_{i+1}e_i=e_i,\hspace{6pt}
e_{i+1}e_ie_{i+1}=e_{i+1},\hspace{6pt} \forall\ 1\leq i\leq n-2,
$$
$$
s_ie_{i+1}e_i=s_{i+1}e_i,\hspace{8pt}
e_{i+1}e_is_{i+1}=e_{i+1}s_i,\hspace{8pt} \forall\ 1\leq i\leq n-2.
$$
It is well-known that $\mathfrak{B}_{n,\Z}$ is a free $\Z$-module of rank $(2n-1)!!=(2n-1)\cdot(2n-3)\cdots 3\cdot 1$. For any field $K$, we define
$\bb_{n,K}:=K\otimes_{\Z}\bb_{n,\Z}$.

Alternatively, the Brauer algebra $\mathfrak{B}_{n,K}$ can be defined in a diagrammatic manner \cite{Brauer}.
Recall that a Brauer $n$-diagram is a graph with $2n$ vertices arranged in two rows (each of $n$ vertices) and $n$ edges such that each
vertex is incident to exactly one edge. Then $\bb_{n,K}$ can be defined as the $K$-linear space with basis the set
${\rm Bd}_n$ of all the Brauer $n$-diagrams. The multiplication of two Brauer $n$-diagrams $D_1$ and $D_2$ is defined by the concatenation of $D_1$ and $D_2$ as
follows: placing $D_1$ above $D_2$, identifying the vertices in the bottom row of $D_1$ with the vertices in the top row of $D_2$, removing
the interior loops in the concatenation and obtaining the composite Brauer $n$-diagram $D_1\circ D_2$, writing $n(D_1,D_2)$ the number of interior loops,
we then define the multiplication $D_1\cdot D_2:=(-2m)^{n(D_1,D_2)}D_1\circ D_2$.

For a Brauer $n$-diagram, we label the vertices in the top row by $1,2,\ldots,n$ from
left to right and the vertices in the bottom row by $\overline{1},\overline{2},\ldots,\overline{n}$ also from left
to right. The two definitions of Brauer algebra $\bb_{n,K}$ can be identified as follows:
\begin{align*} 
\begin{matrix}
\begin{tikzpicture}
\node at (-7.4,-0.5) {$\displaystyle {s_i}= \color{black}$};
\node at (-6.3,-0.4) {$\displaystyle \cdots \color{black}$};
\draw (-6.8,0) -- (-6.8,-0.8);
\draw (-5.8,0) -- (-5.8,-0.8);
\draw (-5,0) -- (-4.2,-0.8);
\draw (-4.2,0) -- (-5,-0.8);
\draw (-3.4,0) -- (-3.4,-0.8);
\draw (-2.4,0) -- (-2.4,-0.8);
\filldraw [black] (-6.8,0) circle (1.2pt);
\filldraw [black] (-6.8,-0.8) circle (1.2pt);
\filldraw [black] (-5.8,0) circle (1.2pt);
\filldraw [black] (-5.8,-0.8) circle (1.2pt);
\filldraw [black] (-5.0,0) circle (1.2pt);
\filldraw [black] (-5.0,-0.8) circle (1.2pt);
\filldraw [black] (-4.2,0) circle (1.2pt);
\filldraw [black] (-4.2,-0.8) circle (1.2pt);
\filldraw [black] (-3.4,0) circle (1.2pt);
\filldraw [black] (-3.4,-0.8) circle (1.2pt);
\filldraw [black] (-2.4,0) circle (1.2pt);
\filldraw [black] (-2.4,-0.8) circle (1.2pt);
\node at (-2.9,-0.4) {$\displaystyle \cdots \color{black}$};
\node at (-5.0,-1.1) {\scalefont{0.8}$\displaystyle \overline{i} \color{black}$};
\node at (-5.0,0.3) {\scalefont{0.8}$\displaystyle i \color{black}$};
\node at (-4.2,0.3) {\scalefont{0.8}$\displaystyle i+1 \color{black}$};
\node at (-4.2,-1.1) {\scalefont{0.8}$\displaystyle \overline{i+1} \color{black}$};
\node at (-1.4,-0.4) {$\displaystyle \text{and} \color{black}$};
\node at (-0.3,-0.5) {$\displaystyle e_i= \color{black}$};
\node at (0.8,-0.4) {$\displaystyle \cdots \color{black}$};
\draw (0.3,0) -- (0.3,-0.8);
\draw (1.3,0) -- (1.3,-0.8);
\draw (2.1,0)  .. controls (2.3,-0.2) and (2.7,-0.2) .. (2.9,0);
\draw (2.1,-0.8)  .. controls (2.3,-0.6) and (2.7,-0.6) .. (2.9,-0.8);
\draw (3.7,0) -- (3.7,-0.8);
\draw (4.7,0) -- (4.7,-0.8);
\filldraw [black] (0.3,0) circle (1.2pt);
\filldraw [black] (0.3,-0.8) circle (1.2pt);
\filldraw [black] (1.3,0) circle (1.2pt);
\filldraw [black] (1.3,-0.8) circle (1.2pt);
\filldraw [black] (2.1,0) circle (1.2pt);
\filldraw [black] (2.1,-0.8) circle (1.2pt);
\filldraw [black] (2.9,0) circle (1.2pt);
\filldraw [black] (2.9,-0.8) circle (1.2pt);
\filldraw [black] (3.7,0) circle (1.2pt);
\filldraw [black] (3.7,-0.8) circle (1.2pt);
\filldraw [black] (4.7,0) circle (1.2pt);
\filldraw [black] (4.7,-0.8) circle (1.2pt);
\node at (4.2,-0.4) {$\displaystyle \cdots \color{black}$};
\node at (2.1,-1.1) {\scalefont{0.8}$\displaystyle \overline{i} \color{black}$};
\node at (2.1,0.3) {\scalefont{0.8}$\displaystyle i \color{black}$};
\node at (2.9,0.3) {\scalefont{0.8}$\displaystyle i+1 \color{black}$};
\node at (2.9,-1.1) {\scalefont{0.8}$\displaystyle \overline{i+1} \color{black}$};
\end{tikzpicture}
\end{matrix}
\end{align*}


Let $K$ be an infinite field. By \cite{Brauer, DP, DDH}, there is a Brauer-Schur-Weyl duality between the symplectic group $Sp_{2m}(K)$ and the Brauer algebra $\bb_{n,K}$ on certain tensor space. To recall the result we need some more notations. Let $V_\Z$ be a free $\Z$-module of rank $2m$.  For each integer
$1\leq i\leq 2m$, set $i':=2m+1-i$. Let $\{v_i\}_{i=1}^{2m}$ be a $\Z$-basis of $V$. Let $\langle\ ,\ \rangle$ be a skew symmetric bilinear form  on $V_\Z$ such that
$$\langle v_i, v_{j}\rangle=0=\langle v_{i'},
v_{j'}\rangle,\,\,\,\langle v_i, v_{j'}\rangle=\delta_{ij}=-\langle
v_{j'}, v_{i}\rangle,\quad\forall\,\,1\leq i, j\leq m. $$
For each integer $1\leq i\leq 2m$, we define
\begin{equation}\label{dualbases} v_i^{\ast}=\begin{cases}
v_{i'}, &\text{if $1\leq i\leq m$;}\\
-v_{i'}, &\text{if $m+1\leq i\leq 2m$.}
\end{cases}
\end{equation}
Then $\{v_i\}_{i=1}^{2m}$ and $\{v_{j}^{\ast}\}_{j=1}^{2m}$ are dual
bases of $V_\Z$ in the sense that $\langle v_i,
v_j^{\ast}\rangle=\delta_{ij}$ for any $i,j$. We define $V:=K\otimes_{\Z}V_\Z$ and abbreviate $1_K\otimes_{\Z}v_i$ by $v_i$ for each $1\leq i\leq 2m$.
There is a natural right action of $\bb_{n,K}$ on $V^{\otimes n}$ which is defined on generators by
$$\begin{aligned} (v_{i_1}\otimes\cdots\otimes v_{i_n})s_j
&:=-(v_{i_1}\otimes\cdots\otimes v_{i_{j-1}}\otimes
v_{i_{j+1}}\otimes v_{i_{j}}\otimes v_{i_{j+2}}
\otimes\cdots \otimes v_{i_n}),\\
(v_{i_1}\otimes\cdots\otimes v_{i_n})e_j
&:=\epsilon_{i_j,i_{j+1}}v_{i_1}\otimes\cdots\otimes v_{i_{j-1}}\otimes
\biggl(\sum_{k=1}^{2m}v_{k}^{\ast}\otimes v_k\biggr)\otimes
v_{i_{j+2}}\otimes\cdots \otimes v_{i_n},\end{aligned}
$$
where for any $i, j\in\bigl\{1,2,\cdots,2m\bigr\}$, $$
\epsilon_{i,j}:=\begin{cases} 1 &\text{if $j=i'$ and $i<j$,}\\
-1 &\text{if $j=i'$ and $i>j$,}\\
0 &\text{otherwise.}\end{cases}
$$
The above right action of $\bb_{n,K}$ on $V^{\otimes n}$ commutes with the natural left diagonal action of the symplectic group $Sp(V)\cong Sp_{2m}(K)$.

Let $k\in\N$. A partition of $k$ is a non-increasing sequence of non-negative integers $\lam=(\lam_1,\lam_2,\cdots)$ which sum to $k$. We write $\lam\vdash k$.
If $\lam\vdash k$ then we set $\ell(\lam):=\max\{t\geq 1|\lam_t\neq 0\}$. The following results are often referred as Brauer-Schur-Weyl duality of type $C$.

\begin{thm}\text{(\cite{Brauer, DP, DDH})}\label{ddh} Assume $K$ is an algebraically closed field. The following two natural homomorphisms are both surjective: $$
\varphi_K: (\bb_{n,K})^{\op}\rightarrow\End_{KSp(V)}(V^{\otimes n}),\quad\,\,
\psi_K: KSp(V)\rightarrow\End_{\bb_{n,K}}(V^{\otimes n}) .
$$
If $m\geq n$ then $\varphi_K$ is an isomorphism. Furthermore, if $K=\C$, then there is a $(\C Sp(V),\bb_{n,\C})$-bimodule decomposition: $$
V^{\otimes n}=\bigoplus_{f=0}^{[n/2]}\bigoplus_{\substack{\lam\vdash n-2f\\ \ell(\lam)\leq m}}\Delta(\lam)\otimes D(f,\lam),
$$
where $\Delta(\lam)$ and $D(f,\lam)$ denote the irreducible $\C Sp(V)$-module corresponding to $\lam$ and the irreducible $\bb_{n,\C}$-module corresponding to $(f,\lam)$ respectively.
\end{thm}

\begin{dfn} We call the endomorphism algebra $S_K^{sy}(m,n):=\End_{\bb_{n,K}}(V^{\otimes n})$ the symplectic Schur algebra.
\end{dfn}

By \cite{Do1, Do2, Oe}, we know that the symplectic Schur algebra is a quasi-hereditary algebra over $K$. Applying Theorem \ref{ASthm}, we can get the following corollary.

\begin{cor}\label{sym1} There exists an injective left $\add(V^{\otimes n})$-approximation $0\rightarrow S_K^{sy}(m,n)\overset{\delta}{\rightarrow}(V^{\otimes n})^{\oplus r}$, where $r\in\N$, which can be continued to an exact sequence
$0\rightarrow S_K^{sy}(m,n)\overset{\delta}{\rightarrow}(V^{\otimes n})^{\oplus r}\overset{\varepsilon}{\rightarrow} (V^{\otimes n})^{\oplus s}$ for some $s\in\N$. In particular, the $V^{\otimes n}$-dominant dimension of $S_K^{sy}(m,n)$ is at least two.
\end{cor}
Alternatively, the above corollary can also be deduced as a direct consequence of Theorem \ref{mainthm1} where Stokke has proved in \cite{Sto} that each Weyl module $\Delta(\lam)$ can be embedded into $V^{\otimes n}$.

There is another version of Brauer-Schur-Weyl duality for dual partially harmonic tensors which was investigated in \cite{Hu3}. Henceforth, we assume that $K$ is an algebraically closed field unless otherwise stated. For each integer $f$ with $0\leq f\leq [n/2]$, let $\mathfrak{B}^{(f)}_{n,K}$ be the two-sided ideal of $\bb_{n,K}$ generated by
$e_1e_3\cdots e_{2f-1}$. By convention, $\mathfrak{B}^{(0)}_{n,K}=\bb_{n,K}$ and
$\mathfrak{B}^{([n/2]+1)}_{n,K}=0$. This gives rise to a two-sided ideals filtration of $\bb_{n,K}$ as follows:
$$
\bb_{n,K}=\mathfrak{B}^{(0)}_{n,K}\supset \mathfrak{B}^{(1)}_{n,K} \supset\cdots
\supset \mathfrak{B}^{([n/2])}_{n,K} \supset 0.
$$

Set $$
\mathcal{HT}_f^{\otimes n}:=\bigl\{v\in V^{\otimes n}\bb_{n,K}^{(f)}\bigm|vx=0,\,\,\forall\,x\in\bb_{n,K}^{(f+1)}\bigr\} .
$$
This space is called (cf. \cite{GW}, \cite{Mali}) the space of partially harmonic tensors of valence $f$ and plays an important role in the study of invariant theory of symplectic groups. It was proved in \cite[1.6]{Hu3} that there is a $(KSp(V),\bb_{n,K}/\bb_{n,K}^{(f+1)})$-bimodule isomorphism \begin{equation}\label{biIso}
\bigl(\mathcal{HT}_f^{\otimes n}\bigr)^*\cong V^{\otimes n}\bb_{n,K}^{(f)}/V^{\otimes n}\bb_{n,K}^{(f+1)},
\end{equation}
and the dimension of $V^{\otimes n}\bb_{n,K}^{(f)}/V^{\otimes n}\bb_{n,K}^{(f+1)}$ is independent of the ground field $K$. For this reason, we call any element in $V^{\otimes n}\bb_{n,K}^{(f)}/V^{\otimes n}\bb_{n,K}^{(f+1)}$ the dual partially harmonic tensor. The natural left action of $KSp(V)$ on $V^{\otimes n}/V^{\otimes n}\bb_{n,K}^{(f)}$ commutes with the right action of $\bb_{n,K}/\bb_{n,K}^{(f)}$ on $V^{\otimes n}/V^{\otimes n}\bb_{n,K}^{(f)}$. Thus we have two natural algebra homomorphisms: $$\begin{aligned}
&\varphi_{f,K}: (\bb_{n,K}/\bb_{n,K}^{(f)})^{\rm{op}}\rightarrow\End_{KSp(V)}(V^{\otimes n}/V^{\otimes n}\bb_{n,K}^{(f)}),\\
&\psi_{f,K}: KSp(V)\rightarrow\End_{\bb_{n,K}/\bb_{n,K}^{(f)}}(V^{\otimes n}/V^{\otimes n}\bb_{n,K}^{(f)}) .
\end{aligned}
$$

\begin{thm}\text{(\cite[1.8]{Hu3})}\label{varphif}  Assume $K$ is an algebraically closed field. Let $0\leq f\leq [n/2]$ be an integer. Then $\dim\End_{KSp(V)}(V^{\otimes n}/V^{\otimes n}\bb_{n,K}^{(f)})$ is independent of the characteristic $K$. Moreover, the natural homomorphism $\varphi_{f,K}$ is surjective.
\end{thm}

\begin{conj}\text{(\cite[5.5]{Hu3})}\label{conj1}  Assume $K$ is an algebraically closed field. Let $0\leq f\leq [n/2]$ be an integer. Then the map $\psi_{f,K}$ is surjective.
\end{conj}

One of our original starting point of this work is our attempt to the proof of the above Conjecture \ref{conj1}. First, we can make some reduction of the above conjecture. It is clear that $$
\End_{\bb_{n,K}/\bb_{n,K}^{(f)}}(V^{\otimes n}/V^{\otimes n}\bb_{n,K}^{(f)})\cong \End_{\bb_{n,K}}(V^{\otimes n}/V^{\otimes n}\bb_{n,K}^{(f)}) .
$$

\begin{dfn}\label{sf} We define $$
S_{f,K}^{sy}(m,n):=\psi_{f,K}(KSp(V))\subseteq\End_{\bb_{n,K}}\bigl(V^{\otimes n}/V^{\otimes n}\bb_{n,K}^{(f)}\bigr).
$$
\end{dfn}

By the main result in \cite{DDH}, the image of $KSp(V)$ in $\End_K(V^{\otimes n})$ is just $S_K^{sy}(m,n)=\End_{\bb_{n,K}}\bigl(V^{\otimes n}\bigr)$.
Let $\pi_{f,K}: \End_{\bb_{n,K}}(V^{\otimes n})\rightarrow\End_{\bb_{n,K}}\bigl(V^{\otimes n}/V^{\otimes n}\bb_{n,K}^{(f)}\bigr)$ be the natural homomorphism. By construction, we have the following commutative diagram: \begin{equation}\label{commdiag1}
\xymatrix{
 KSp(V) \ar@{>>}[r]^{\psi_{K}\hspace{40pt}} \ar@{>>}[d]_{\psi_{f,K}} &
 S_K^{sy}(m,n)=\End_{\bb_{n,K}}\bigl(V^{\otimes n}\bigr)
\ar@{>}[d]_{\pi_{f,K}}\\
 S_{f,K}^{sy}(m,n) \ar@{^{(}->}[r]^{\iota_{f,K}\hspace{35pt}} &
\End_{\bb_{n,K}}\bigl(V^{\otimes n}/V^{\otimes n}\bb_{n,K}^{(f)}\bigr),}
\end{equation}
where the top horizontal map and the left vertical map are both surjective, and the bottom horizontal map is injective. As a result, the map $\pi_{f,K}$ gives rise to a surjection \begin{equation}\label{pif0} \pi_{f,K}: S_K^{sy}(m,n)\twoheadrightarrow S_{f,K}^{sy}(m,n).\end{equation}
The advantage of working with $S_K^{sy}(m,n)$ lies in that we can now allow $K$ to be an arbitrary (not necessarily infinite) field or even an integral domain. We use $\psi'_{f,K}$ to denote the composition of $\pi_{f,K}$ with the natural inclusion $S_{f,K}^{sy}(m,n)\hookrightarrow\End_{\bb_{n,K}}(V_K^{\otimes n}/V_K^{\otimes n}\bb_{n,K}^{(f)})$.
It is easy to see that Conjecture \ref{conj1} is a consequence of the following conjecture.

\begin{conj}\label{conj2}
Let $0\leq f\leq [n/2]$ be an integer and $K$ an arbitrary field. Then the natural map
$\psi'_{f,K}: S_K^{sy}(m,n)\rightarrow\End_{\bb_{n,K}}(V_K^{\otimes n}/V_K^{\otimes n}\bb_{n,K}^{(f)})$ is surjective.
\end{conj}

Let $\bk$ be the algebraic closure of $K$. It is clear that $\psi'_{f,K}$ is surjective if and only if $\psi'_{f,\bk}$ is surjective. Suppose that $\cha K=0$ or $\cha K>n$. Then by \cite[Lemma 5.16]{Ga}, $V_\bk^{\otimes n}$ is a semisimple $\bk Sp(V)$-module. In particular, $\End_{\bk Sp(V_\bk)}(V_\bk^{\otimes n})$ is semisimple. Since the action of $\bb_{n,\bk}$ on $V_\bk^{\otimes n}$ factors through the action of $\End_{\bk Sp(V_\bk)}(V_\bk^{\otimes n})$ on $V_\bk^{\otimes n}$ and the natural homomorphism $\varphi_\bk: \bb_{n,\bk}^{\rm{op}}\rightarrow\End_{\bk Sp(V_\bk)}(V_\bk^{\otimes n})$ is surjective (\cite{DDH}), it follows that the action of $\bb_{n,\bk}$ on $V_\bk^{\otimes n}$ is semisimple too. Therefore, it is easy to see that Conjecture \ref{conj2} holds in this case. Next we shall show that Conjecture \ref{conj2} also holds when
$\cha K>n-f+2m$.

Let $U_\Z(\mathfrak{sp}_{2m})$ be the Kostant $\Z$-form of the universal enveloping algebra of the symplectic Lie algebra $\mathfrak{sp}_{2m}(\C)$. For any field $K$, we define $U_K(\mathfrak{sp}_{2m}):=K\otimes_{\Z}U_\Z(\mathfrak{sp}_{2m})$. By the main result of \cite{Hu2}, we have two surjective algebra homomorphisms: $$
\varphi_K: (\bb_{n,K})^{\op}\twoheadrightarrow\End_{U_K(\mathfrak{sp}_{2m})}(V_K^{\otimes n}),\quad\,\,
\psi_K: U_K(\mathfrak{sp}_{2m})\twoheadrightarrow S_{K}^{sy}(m,n)=\End_{\bb_{n,K}}(V_K^{\otimes n}) .
$$
As a result, $S_K^{sy}(m,n)$ has the double centralizer property with respect to $V_K^{\otimes n}$. Now applying Theorem \ref{ASthm}, we have an exact sequence of $S_K^{sy}(m,n)$-module homomorphisms: \begin{equation}\label{zzz}
0\rightarrow S_K^{sy}(m,n)\overset{\delta_K}{\rightarrow}(V_K^{\otimes n})^{\oplus r}\overset{\varepsilon_K}{\rightarrow} (V_K^{\otimes n})^{\oplus s}, \end{equation}
where $r,s\in\N$, and $\delta_K$ is a left $\add(V_K^{\otimes n})$-approximation of $S_K^{sy}(m,n)$.

Let $1_K$ be the unit element of $S_K^{sy}(m,n)$. We can write $$
\delta_K(1)=u_1\oplus\cdots\oplus u_r,
$$
where $u_i\in V_K^{\otimes n}$ for each $1\leq i\leq r$. Suppose that $a\in\Ker\pi_{f,K}=\Hom_{\bb_{n,K}}(V_K^{\otimes n},V_K^{\otimes n}\bb_{n,K}^{(f)})$. It follows that $$
\delta_K(a)=au_1\oplus\cdots\oplus au_r .
$$
Now $\pi_{f,K}(a)=0$ means that $aV_K^{\otimes n}\subseteq V_K^{\otimes n}\bb_{n,K}^{(f)}$.
As a result, we get that $$
\delta_K(a)=au_1\oplus\cdots\oplus au_r\in\bigl(V_K^{\otimes n}\bb_{n,K}^{(f)}\bigr)^{\oplus r} .
$$
This shows that $\delta_K$ induces a well-defined homomorphism $$\delta_{f,K}: S_{f,K}^{sy}(m,n)\rightarrow (V_K^{\otimes n}/V_K^{\otimes n}\bb_{n,K}^{(f)})^{\oplus r}.$$

\begin{lem}\label{approx} With the notations as above, the integers $r,s$, the map $\delta_K$ and the elements $u_1,\cdots,u_r\in V_K^{\otimes n}$ can be chosen such that the map $\delta_{f,K}$ is an injective left $\add(V_K^{\otimes n}/V_K^{\otimes n}\bb_{n,K}^{(f)})$-approximation of $S_{f,K}^{sy}(m,n)$.
\end{lem}

\begin{proof} The integer $r$, the map $\delta_K$ and $u_1,\cdots,u_r$ can be chosen such that $u_1,\cdots,u_r$ is a $K$-linear generator of $V_K^{\otimes n}$. In view of Remark \ref{rem1}, it suffices to show that $\delta_{f,K}$ is injective.

Suppose that $(a\overline{u_1},\cdots,a\overline{u_r})=\delta_K(a)=\delta_{f,K}(\overline{a})=\overline{0}$, where $a\in S_K^{sy}(m,n)$. Since $u_1,\cdots,u_r$ is a $K$-linear generator of $V_K^{\otimes n}$, it follows that $aV_K^{\otimes n}\subseteq V_K^{\otimes n}\bb_{n,K}^{(f)}$, which implies that $\pi_{f,K}(a)=0$ and hence $\iota_{f,K}(\overline{a})=\pi_{f,K}(a)=0$. This proves $\overline{a}=\overline{0}$ as required.
\end{proof}

\begin{dfn} Set $$\begin{aligned}
\Lam&:=\bigl\{\lam=(\lam_1,\cdots,\lam_m)\bigm|\sum_{i=1}^{m}\lam_i=n-2r, 0\leq r\leq[n/2],\lam_i\in\Z,\forall\,1\leq i\leq m\bigr\},\\
\Lam^+&:=\bigl\{\lam=(\lam_1,\cdots,\lam_m)\in\Lam\bigm|\sum_{i=1}^{m}\lam_i=n-2r, 0\leq r\leq[n/2],\lam_1\geq\lam_2\geq\cdots\geq\lam_m\geq 0\bigr\}. \end{aligned} $$
Then $\Lam$ (resp., $\Lam^+$) is the set of all weights (resp., dominant weights) of $V^{\otimes n}$ as $Sp(V)$-module. For any integer $0\leq f\leq [n/2]$, we define \begin{equation}\label{lamf}
\Lam_f^+:=\bigl\{\lam=(\lam_1,\cdots,\lam_m)\in\Lam^+\bigm|\sum_{i=1}^{m}\lam_i=n-2r', f\leq r'\leq [n/2], r'\in\N\bigr\},\quad \Lam_f^c:=\Lam^+\setminus\Lam_f^+ .
\end{equation}
\end{dfn}

Let $E$ be an Euclidian space with standard basis $\eps_1,\cdots,\eps_{2m}$. Let $S:=\{\eps_i-\eps_{i+1},2\eps_m|1\leq i<m\}$, which is a set of simple roots in the root system $\Phi$ of type $C_m$. Let $\Phi^+$ be the corresponding subset of positive roots. We identify each $\lam=(\lam_1,\cdots,\lam_m)\in\Lam$ with $\lam_1\eps_1+\cdots+\lam_m\eps_m$. For any $\lam,\mu\in\Lam$, we define $\lam\geq\mu$ if and only if
$\lam-\mu\in\sum_{\alpha\in S}\N\alpha$.

Recall that $(S_K^{sy}(m,n),\Lam^+,\geq)$ is a quasi-hereditary algebra (\cite{Do1,Do2}).  For each $\lam\in\Lam^+$, we use $\Delta(\lam), \nabla(\lam), L(\lam)$ to denote the standard module, costandard module and simple module labelled by $\lam$ respectively.

Let $\lam,\mu\in\Lam^+$, where $\lam\vdash n-2a, \mu\vdash n-2b$ and $0\leq a<b\leq[n/2]$. We claim that $\lam\not\leq\mu$. In fact, suppose that $\lam\leq\mu$, then there are some non-negative integers $a_1,\cdots,a_m$ such that $$
\mu-\lam=\sum_{i=1}^{m}\mu_i\eps_i-\sum_{i=1}^{m}\lam_i\eps_i=a_1(\eps_1-\eps_2)+\cdots+a_{m-1}(\eps_{m-1}-\eps_{m})+a_m2\eps_m .
$$
It follows that $$
0>2(a-b)=(n-2b)-(n-2a)=\sum_{i=1}^{m}\mu_i-\sum_{i=1}^{m}\lam_i=2a_m\geq 0,
$$
which is a contradiction. This proves our claim which is the following lemma\footnote{We corrected a small error here in the argument in the proof of \cite[Lemma 3.7]{Hu3}.}.

\begin{lem}\text{(\cite[Lemma 3.7]{Hu3})}\label{order} Let $\lam,\mu\in\Lam^+$, where $\lam\vdash n-2a, \mu\vdash n-2b$ and $0\leq a<b\leq[n/2]$. Then $\lam\not\leq\mu$.
\end{lem}

Recall that $\bigl(V_K^{\otimes n}/V_K^{\otimes n}\bb_{n,K}^{(f)}\bigr)^{\ast}\hookrightarrow\bigl(V_K^{\otimes n}\bigr)^{\ast}\cong V_K^{\otimes n}$. Henceforth, we shall use this embedding to identify  $\bigl(V_K^{\otimes n}/V_K^{\otimes n}\bb_{n,K}^{(f)}\bigr)^{\ast}$ as a $K$-subspace of  $V_K^{\otimes n}$. Since the isomorphism $\bigl(V_K^{\otimes n}\bigr)^\ast\cong V_K^{\otimes n}$ is a right $\bb_{n,K}$-module isomorphism, it follows that for any $x\in V_K^{\otimes n}$, $x\in \bigl(V_K^{\otimes n}/V_K^{\otimes n}\bb_{n,K}^{(f)}\bigr)^{\ast}$ if and only if $x\bb_{n,K}^{(f)}=0$.

\begin{lem}\label{keylem4} Suppose that $\cha K>\min\{n-f+2m,n\}$. Then $V_K^{\otimes n}\bb_{n,K}^{(f)}\cap\bigl(V_K^{\otimes n}/V_K^{\otimes n}\bb_{n,K}^{(f)}\bigr)^{\ast}=0$. In particular, $$
V_K^{\otimes n}=V_K^{\otimes n}\bb_{n,K}^{(f)}\oplus\bigl(V_K^{\otimes n}/V_K^{\otimes n}\bb_{n,K}^{(f)}\bigr)^{\ast} .
$$
\end{lem}

\begin{proof} Without loss of generality, we can assume $K=\bk$ is an algebraically closed field. To prove the first part of the lemma, it suffices to show that for any $\lam\in\Lam_f^c$ and $\mu\in\Lam_f^+$, $L(\lam), L(\mu)$ are not in the same block as $KSp(V)$-module.

Suppose that this is not the case. Set $p:=\cha K$. If $p>n$ then by \cite[Theorem 5.16]{Ga}, $V_K^{\otimes n}$ is a semisimple $KSp(V)$-module. In this case the lemma clearly holds. Henceforth, we assume that $p>n-f+2m$. By the linkage principal (\cite[Part II, Chapter 6]{Ja}) and Lemma \ref{order}, we must be able to find $\lam=\nu_r\in\Lam_f^c$, $\mu=\nu_0\in\Lam_f^+$ and dominant weights $\nu_1,\nu_2,\cdots,\nu_{r-1}$ of $Sp(V)$, such that $\nu_0\uparrow\nu_1\uparrow\cdots\uparrow\nu_{r-1}\uparrow\nu_r$, where $\mu\uparrow\lam$ means that $\mu<\lam$, and there exist affine reflections $s_{\beta_1,n_1p},\cdots,s_{\beta_r,n_rp}$ (where $\beta_1,\cdots,\beta_r\in\Phi^+$, $n_1,\cdots,n_r\in\Z$) such that $$
\nu_0=\mu<\nu_1=s_{\beta_1,n_1p}\cdot\mu<\nu_2=s_{\beta_2,n_2p}s_{\beta_1,n_1p}\cdot\mu<\cdots<\nu_r=s_{\beta_r,n_rp}\cdots s_{\beta_1,n_1p}\cdot\mu=\lam ,
$$
where $$
s_{\beta_i,n_ip}\cdot\mu:=\mu-\<\mu+\rho,\beta_i^{\vee}\>\beta_i+n_ip\beta_i ,
$$
and $\rho=\sum_{i=1}^{m}(m-i+1)\eps_i$ is the half sum of all positive roots. In particular, $\nu_2,\cdots,\nu_{r-1}\in\Lam^+$.

If $\beta_1=2\eps_i$ for some $1\leq i\leq m$, then we have $$
s_{\beta_1,n_1p}\cdot\mu=\sum_{s\neq i}\mu_s\eps_s+(2pn_1-\mu_i-2(m+1-i))\eps_i\in\Lam^+ ,
$$
which implies $n_1\in\Z^{\geq 1}$. As $\mu\in\Lam_f^+$, we have $$
2pn_1-\mu_i-2(m-i+1)\geq 2p-\mu_i-2m>2(n+2m-f)-\mu_i-2m=n-2f-\mu_i+n+2m>n,
$$
which is impossible because $s_{\beta_1,n_1p}\cdot\mu\in\Lam^+$.

If $\beta_1=\eps_i+\eps_j$ for some $1\leq i<j\leq m$, then we have  $$
s_{\beta_1,n_1p}\cdot\mu=\sum_{s\neq i,j}\mu_s\eps_s+(pn_1-\mu_j-2m+i+j-2)\eps_i+(pn_1-\mu_i-2m+i+j-2)\eps_j\in\Lam^+ .
$$
Since $|\mu_i-i+j|<n-2f+m$ and $p>n+2m-f$, it follows that if $n_1\geq 1$ then $$
(pn_1-\mu_j-2m+i+j-2)+(pn_1-\mu_i-2m+i+j-2)=2pn_1-(\mu_i+\mu_j)-4m+2(i+j-2)>n,
$$
which is impossible. Thus we must have that $n_1=0$.

If $\beta_1=\eps_i-\eps_j$ for some $1\leq i<j\leq m$, then we have  $$
s_{\beta_1,n_1p}\cdot\mu=\sum_{s\neq i,j}\mu_s\eps_s+(pn_1+\mu_j+i-j)\eps_i+(-pn_1+\mu_i-i+j)\eps_j\in\Lam^+ .
$$
Since $|\mu_i-i+j|<n-2f+m$ and $p>n+2m-f$, it follows that $n_1=0$ again.

As a consequence, we can deduce that $\nu_1=s_{\beta_1,0}\cdot\mu\in\Lam_f^+$. Now replacing
$\mu$ with $s_{\beta_1,n_1p}\cdot\mu$ and continuing the same argument, we shall finally show that $\nu_2,\cdots,\nu_r\in\Lam_f^+$, which contradicts to our assumption that $\lam=\nu_r\in\Lam_f^c$. This completes the proof of the first part of the lemma.

As $\dim V_K^{\otimes n}\bb_{n,K}^{(f)}+\dim\bigl(V_K^{\otimes n}/V_K^{\otimes n}\bb_{n,K}^{(f)}\bigr)^{\ast}=\dim V_K^{\otimes n}$, it is clear that
the second part of the lemma follows from the first part of the lemma.
\end{proof}

Recall the definition of $\varepsilon_K$ in (\ref{zzz}). We can rewrite the homomorphism $\varepsilon_K: (V_K^{\otimes n})^{\oplus r}\rightarrow (V_K^{\otimes n})^{\oplus s}$ as follows: $$\begin{aligned}
\varepsilon_K: (V_K^{\otimes n})^{\oplus r} & \rightarrow (V_K^{\otimes n})^{\oplus s}\\
w_1\oplus\cdots\oplus w_r&\mapsto \sum_{k=1}^{r}\varepsilon_{k1}(w_k)\oplus\cdots\oplus \sum_{k=1}^{r}\varepsilon_{ks}(w_k),
\end{aligned}
$$
where for each $1\leq k\leq r, 1\leq t\leq s$, $\varepsilon_{kt}\in\End_{U_K(\mathfrak{sp}_{2m})}(V_K^{\otimes n})$.

Applying Theorem \ref{ddh}, we can find $b_{kt}\in\bb_{n,K}$ such that $\varphi_K(b_{kt})=\varepsilon_{kt}$ for each $1\leq k\leq r, 1\leq t\leq s$. Since $\bb_{n,K}^{(f)}\bb_{n,K}\subseteq\bb_{n,K}^{(f)}$, it follows that $\varepsilon_{kt}(V_K^{\otimes n}\bb_{n,K}^{(f)})\subseteq V_{K}^{\otimes n}\bb_{n,K}^{(f)}$. As a consequence, we see that $\varepsilon_K$ induces a $S_K^{sy}(m,n)$-module homomorphism $$
{\eps}_{f,K}:  (V_K^{\otimes n}/V_K^{\otimes n}\bb_{n,K}^{(f)})^{\oplus r}\rightarrow (V_K^{\otimes n}/V_K^{\otimes n}\bb_{n,K}^{(f)})^{\oplus s} .
$$

Now we can give the proof of the fourth main result of this paper.

\medskip
\noindent
{\textbf{Proof of Theorem \ref{mainthm4}}:}  By Lemma \ref{approx}, we know that $\delta_{f,K}$ is injective. It follows from $\varepsilon_K\circ\delta_K=0$ that ${\varepsilon}_{f,K}\circ\delta_{f,K}=0$. We claim that $\Ker\varepsilon_{f,K}=\im(\delta_{f,K})$. t suffices to show that $\Ker\varepsilon_{f,K}\subseteq\im(\delta_{f,K})$.

Let $w_1,\cdots,w_r\in V_K^{\otimes n}$ such that $\varepsilon_{f,K}(\overline{w_1}\oplus\cdots\oplus\overline{w_r})=\overline{0}$. We set $x:={w_1}\oplus\cdots\oplus{w_r}$. Using Lemma \ref{keylem4}, we can decompose $x$ as $x=y+z$ such that $y\in \Bigl(\bigl(V_K^{\otimes n}/V_K^{\otimes n}\bb_{n,K}^{(f)}\bigr)^{\ast}\Bigr)^{\oplus r}$ and $z\in \bigl(V_K^{\otimes n}\bb_{n,K}^{(f)}\bigr)^{\oplus r}$. It follows that $$
\eps_K(y)=\eps_K(x-z)=\eps_K(x)-\eps_K(z)\in \bigl(V_K^{\otimes n}\bb_{n,K}^{(f)}\bigr)^{\oplus s}.
$$

On the other hand, since $\bigl(V_K^{\otimes n}/V_K^{\otimes n}\bb_{n,K}^{(f)}\bigr)^{\ast}$ is filtered by some $\Delta(\lam)$ with $\lam\in\Lam_f^c$, while
$V_K^{\otimes n}\bb_{n,K}^{(f)}$ is filtered by some $\nabla(\mu)$ with $\mu\in\Lam_f^+$. It follows from that $$
\Hom_{S_K^{sy}(m,n)}\Bigl(\bigl(V_K^{\otimes n}/V_K^{\otimes n}\bb_{n,K}^{(f)}\bigr)^{\ast},V_K^{\otimes n}\bb_{n,K}^{(f)}\Bigr)=0 . $$
Hence by Lemma \ref{keylem4} we must have that $\eps_K(y)=0$. Applying Corollary \ref{sym1} and (\ref{zzz}), we can deduce that
$y=\delta_K(a)$ for some $a\in S_K^{sy}(m,n)$. It follows that $$
x+\bigl(V_K^{\otimes n}\bb_{n,K}^{(f)}\bigr)^{\oplus r}=y+\bigl(V_K^{\otimes n}\bb_{n,K}^{(f)}\bigr)^{\oplus r}=\delta_{f,K}(a+\Ker\pi_{f,K}),
$$
which implies our claim. Therefore, by Theorem \ref{ASthm}, $S_{f,K}^{sy}(m,n)$ has the double centralizer property with respect to $V_K^{\otimes n}/V_K^{\otimes n}\bb_{n,K}^{(f)}$.

Finally, by \cite[Theorem 1.8]{Hu3}, we know that the natural map $$
\bb_{n,K}^{\rm{op}}\rightarrow\End_{KSp(V)}(V_K^{\otimes n}/V_K^{\otimes n}\bb_{n,K}^{(f)})=\End_{S_K^{sy}(m,n)}(V_K^{\otimes n}/V_K^{\otimes n}\bb_{n,K}^{(f)})=\End_{S_{f,K}^{sy}(m,n)}(V_K^{\otimes n}/V_K^{\otimes n}\bb_{n,K}^{(f)})
$$
is surjective. Combining this with the double centralizer property of $S_{f,K}^{sy}(m,n)$ with respect to $V_K^{\otimes n}/V_K^{\otimes n}\bb_{n,K}^{(f)}$ we see that Conjecture \ref{conj2} and hence Conjecture \ref{conj1} hold in this case. This completes the proof of the theorem.
\medskip

\begin{cor} Suppose that $\cha K>\min\{n-f+2m,n\}$. Then Conjecture \ref{conj2} and hence Conjecture \ref{conj1} hold in this case.
\end{cor}


\bigskip
\begin{thebibliography}{}



\bibitem{AHLU1} {\sc I.\'{A}goston, D. Happel, E. Luk\'{a}cs, and L.~Unger}, {\em Standardly stratified algebras and tilting}, J. Alg., \textbf{226}(1)
(2000), 144--160.



\bibitem{ADL} {\sc I. \'{A}goston, V.~Dlab, E.~Luk\'{a}cs}, {\em Stratified algebras}, C. R. Math. Acad. Sci. Soc. R. Can., \textbf{20}(1) (1998), 22--28.

\bibitem{AK} {\sc S.~Ariki and K.~Koike}, {\em A Hecke algebra of $(\Z/r\Z)\wr\Sym_n$ and construction of its irreducible
representations}, Adv. Math., \textbf{106} (1994), 216--243.

\bibitem{AS} {\sc M.~Auslander and {\O}.~Solberg}, {\em Relative homology and representation theory. III. Cotilting
modules and Wedderburn correspondences}, Comm. Alg., \textbf{21} (1993), 3081--3097.

\bibitem{Bac} {\sc E.~Backlin}, {\em Koszul duality for parabolic and singular category $\co$}, Represent. Theory, \textbf{3}(1999), 139--152.

\bibitem{BGS}
{\sc A.~Beilinson, V.~Ginzburg and W.~Soergel}, {\em Koszul duality patterns in representation theory}, {J. Amer. Math. Soc.}, \textbf{9} (1996) 473--527.


\bibitem{Brauer} {\sc R.~Brauer}, {\em On algebras which are connected with semisimple
continuous groups}, Ann. of Math., \textbf{38} (1937), 857--872.

\bibitem{Brown} {\sc W.P.~Brown}, {\em An algebra related to the orthogonal group}, Michigan Math. J., \textbf{3} (1955-1956), 1--22.

\bibitem{BK} {\sc J.~Brundan and A.~Kleshchev}, {\em Schur-Weyl duality for higher levels}, Selecta Math. (N.S.), \textbf{14} (2008), 1--57.



\bibitem{CL} {\sc R.W.~Carter and G.~Lusztig}, {\em On the modular representations of general linear and symmetric
groups}, Math. Zeit., \textbf{136} (1974), 193--242.

\bibitem{CP} {\sc V.~Chari and A.~Pressley}, {\em  A guide to quantum groups}, Cambridge University Press, Cambridge,
1994.

\bibitem{CP2} \leavevmode\vrule height 2pt depth -1.6pt width 23pt, {\em Quantum affine algebras and affine Hecke algebras}, Pacific. J. Math., \textbf{174} (1996), 295--326.

\bibitem{CPS0} {\sc E.~Cline, B.~Parshall and L.~Scott}, {\em Finite-dimensional algebras and highest weight categories},
J. Reine Angew. Math., \textbf{391} (1988), 85--99.

\bibitem{CPS} \leavevmode\vrule height 2pt depth -1.6pt width 23pt,  {\em Stratifying endomorphism algebras}, Mem. Amer. Math. Soc., \textbf{124} (1996), no. 591.

\bibitem{CVM1} {\sc A.~Cox, M.~De Visscher and P.~Martin}, {\em The blocks of the Brauer
algebra in characteristic zero}, Represent. Theory, \textbf{13} (2009), 272--308.

\bibitem{CVM2} \leavevmode\vrule height 2pt depth -1.6pt width 23pt, {\em A geometric characterisation
of the blocks of the Brauer algebra}, J. Lond. Math. Soc., \textbf{80} (2009), 471--494.

\bibitem{DP} {\sc C.~De Concini and C.~Procesi}, {\em A characteristic free approach to
invariant theory}, Adv. Math. \textbf{21} (1976), 330--354.

\bibitem{DS} {\sc C.~De Concini and E.~Strickland}, {\em Traceless tensors and the symmetric groups},
J. Alg., \textbf{61} (1979), 112--128.

\bibitem{DDF} {\sc B.~Deng, J.~Du and Q.~Fu}, {\em A Double Hall algebra approach to affine quantum Schur--Weyl Theory}, London Mathematical Society Lecture Note Series, Volume \textbf{401}, Cambridge University Press, 2012.

\bibitem{DDH} {\sc R.~Dipper, S.~Doty and J.~Hu}, {\em Brauer algebras, symplectic Schur algebras and
Schur-Weyl duality}, Trans. Amer. Math. Soc., \textbf{360} (2008), 189--213.

\bibitem{DJ1} {\sc R.~Dipper and G.D.~James}, {\em The $q$-Schur algebra}, Proc. Lond. Math.
Soc., \textbf{59}(3) (1989), 23--50.

\bibitem{DJM} {\sc R.~Dipper, G.D.~James and A.~Mathas}, {\em Cyclotomic $q$-Schur algebras}, Math. Z., \textbf{229} (1998), 385--416.

\bibitem{D1} {\sc V.~Dlab}, {\em Properly stratified algebras}, C. R. Acad. Sci. Paris S\'{e}r. I Math., \textbf{331}(3) (2000), 191--196.

\bibitem{DR1} {\sc V. Dlab, C.M.~Ringel}, {\em Quasi-hereditary algebras}, Illinois J. Math., \textbf{33}(2) (1989), 280--291.

\bibitem{Do1} {\sc S.~Donkin}, {\em On Schur algebras and related algebras I}, J. Alg., \textbf{104} (1986), 310--328.

\bibitem{Do2} \leavevmode\vrule height 2pt depth -1.6pt width 23pt, {\em Good filtrations of rational modules for reductive groups}, Arcata Conf. on
Representations of Finite Groups. Proceedings of Symp. in Pure Math., \textbf{47} (1987), 69--80.

\bibitem{Do3}\leavevmode\vrule height 2pt depth -1.6pt width 23pt, {\em The  {$q$}-{S}chur algebra}, London Mathematical Society Lecture Note Series,
  {\bf 253}, Cambridge University Press, Cambridge, 1998.

\bibitem{Do4}\leavevmode\vrule height 2pt depth -1.6pt width 23pt, {\em On Schur algebras and related algebras VI: Some remarks on rational and classical Schur
algebras},  J. Alg., \textbf{405} (2014), 92--121.


\bibitem{Do5}\leavevmode\vrule height 2pt depth -1.6pt width 23pt, {\em Some remarks on tilting modules, double centralisers and Lie modules}, J. Alg., \textbf{526} (2019), 188--210.

\bibitem{DoranWH} {\sc W.F.~Doran, D.B.~Wales and P.~Hanlon}, {\em On the semisimplicity
of the Brauer centralizer algebras}, J. Alg., \textbf{211} (1999), 647--685.

\bibitem{DH} {\sc S.~Doty and J.~Hu}, {\em Schur-Weyl duality for orthogonal groups}, Proc. London Math.
Soc., \textbf{98} (2009), 679--713.

\bibitem{Du} {\sc J.~Du}, {\em A note on quantized Weyl reciprocity at roots of unity}, Algebra Colloq., \textbf{2} (1995), 363--372.

\bibitem{DPS} {\sc J.~Du, B.~Parshall and L.~Scott}, {\em Quantum Weyl reciprocity and tilting modules}, Comm.
Math. Phys., \textbf{195} (1998), 321--352.

\bibitem{En} {\sc J.~Enyang}, {\em Cellular bases for the Brauer and Birman-Murakami-Wenzl algebras},
J. Alg., \textbf{281} (2004), 413--449.

\bibitem{F} {\sc A.~Frisk}, {\em Dlab's theorem and tilting modules for stratified algebras}, J. Alg., \textbf{314}(2) (2007), 507--537.

\bibitem{Ga} {\sc F.~Gavarini}, {\em On the radical of Brauer algebras}, Mathematische Zeitschrift, \textbf{260} (2008), 673--697.

\bibitem{GGOR} {\sc V.~Ginzburg, N.~Guay, E.~Opdam and R.~Rouquier}, {\em On the category $\mathcal{O}$ for rational Cherednik
algebras}, Invent. math., \textbf{154}, (2003), 617--651.




\bibitem{GW} {\sc R.~Goodman and N.R.~Wallach}, {\em Representations and invariants of classical groups},
Cambridge University Press, 1998.

\bibitem{G}
{\sc J.A.~Green}, {\em Polynomial representations of $GL_n$}, Lect. Notes Math., \textbf{830}, Springer-Verlag, Berlin 1980.

\bibitem{Ha}
{\sc T.~Hayashi}, {\em Quantum deformation of classical groups}, Publ. RIMS. Kyoto Univ., \textbf{28} (1992),
57--81.

\bibitem{HK1}
{\sc A.~Henke and S.~K\"onig}, {\em Schur algebras of Brauer algebras I}, Math. Zeit., \textbf{272} (2012), 729--759.

\bibitem{HK2}
\leavevmode\vrule height 2pt depth -1.6pt width 23pt, {\em Schur algebras of Brauer algebras II}, Math. Zeit., \textbf{276} (2014), 1077--1099.

\bibitem{HP}
{\sc A.~Henke and R.~Paget}, {\em Brauer algebras with parameter $n=2$ acting on tensor space},  Algebras Represent.
Theory, \textbf{11} (2008), 545--575.



\bibitem{Hu2} {\sc J.~Hu}, {\em BMW algebra, quantized coordinate algebra and type $C$ Schur-Weyl duality},
Represent. Theory, \textbf{15} (2010), 1--62.

\bibitem{Hu3} \leavevmode\vrule height 2pt depth -1.6pt width 23pt, {\em Dual partially harmonic tensors and Brauer-Schur-Weyl duality},
Transform. Groups, \textbf{15} (2010), 333--370.

\bibitem{HuXiao} {\sc J.~Hu and Z.~Xiao}, {\em On tensor spaces for Birman-Murakami-Wenzl algebras},
J. Alg., \textbf{324} (2010), 2893--2922.

\bibitem{Hum}
{\sc J.E.~Humphreys}, {\em Representations of Semisimple Lie Algebras in the BGG Category $\mathcal{O}$}, Graduate Studies in Mathematics, vol. \textbf{94}, American Mathematical Society, Providence, RI, 2008.

\bibitem{Ja}
{\sc J.C.~Jantzen}, {\em Representations of algebraic groups, 2nd ed.}, American Mathematical Society, Providence, RI, 2003.


\bibitem{J} {\sc M.~Jimbo}, {\em A q-analogue of $U(\mathfrak{gl}(N + 1))$, Hecke algebra, and the Yang-Baxter equation}, Letters in
Math. Phys., \textbf{11} (1986), 247--252.




\bibitem{KSX} {\sc S.~K\"onig, I.H.~Slung{\aa}rd and C.C.~Xi}, {\em Double centralizer properties, dominant dimension,
and tilting modules}, J. Alg., \textbf{240} (2001), 393--412.

\bibitem{LR} {\sc Z.~Lin and H.~Rui}, {\em Cyclotomic $q$-Schur algebras and Schur-Weyl duality}, in: Representations of algebraic groups, quantum groups, and Lie algebras, 133--155, Contemp. Math., \textbf{413}, Amer. Math. Soc., Providence, RI, 2006.


\bibitem{Mali} {\sc M.~Maliakas}, {\em Traceless tensors and invariants}, Math. Proc.
Camb. Phil. Soc., \textbf{124} (1998), 73--80.

\bibitem{Martin} {\sc P.~Martin}, {\em The decomposition matrices of the Brauer algebra
over the complex field}, Trans. Amer. Math. Soc., \textbf{367} (2015), 1797--1825.

\bibitem{Mat} {\sc A.~Mathas}, {\em Iwahori-Hecke algebras and Schur algebras of the symmetric group}, Univ. Lecture Ser., vol. \textbf{15}, Amer. Math.
Soc., 1999.

\bibitem{M1} {\sc V.~Mazorchuk}, {\em Stratified algebras arising in Lie theory}, in: Representations of Finite Dimensional Algebras
and Related Topics in Lie Theory and Geometry, in: Fields Inst. Commun., vol. \textbf{40}, Amer. Math. Soc., Providence,
RI, 2004, pp. 245--260.

\bibitem{MO} {\sc V.~Mazorchuk and S.~Ovsienko}, {\em Finitistic dimension of properly stratified algebras}, Advances in Mathematics, \textbf{186} (2004),
251--265.

\bibitem{MS}
{\sc V.~Mazorchuk and C.~Stroppel}, {\em Projective-injective modules, Serre functors and symmetric algebras},
J. Reine. Angew. Math., \textbf{616}, (2008), 131--165.

\bibitem{Oe}
{\sc S.~Oehms}, {\em Centralizer coalgebras, FRT-construction, and symplectic monoids}, J. Alg.,
\textbf{244} (2001), 19--44.

\bibitem{Oe2}
 \leavevmode\vrule height 2pt depth -1.6pt width 23pt,  {\em  Symplectic $q$-Schur algebras}, J. Alg., \textbf{304}(2) (2006), 851--905.

\bibitem{RW}
{\sc A.~Ram and H.~Wenzl}, {\em Matrix units for centralizer algebras}, J. Alg., \textbf{145} (1992), 378--395.

\bibitem{R}
{\sc C.M.~Ringel}, {\em The category of modules with good filtrations over a quasi-hereditary
algebra has almost split sequences}, Math. Z., \textbf{208} (1991), 209--223.


\bibitem{Ro}
{\sc A. Rocha-Caridi}, {\em Splitting criteria for $\mathfrak{g}$-modules induced from a parabolic and the Bernstein-Gelfand-Gelfand resolution of a finite-dimensional, irreducible $\mathfrak{g}$-module}, Trans. Amer. Math. Soc., \textbf{262} (1980), 335--366.

\bibitem{Rui} {\sc H.~Rui}, {\em A criterion on the semisimple Brauer algebras},
J. Combin. Theory Ser. A, \textbf{111} (2005), 78--88.

\bibitem{So} {\sc W.~Soergel}, {\em Kategorie $\mathcal{O}$, perverse Garben und Moduln \"{u}ber den Koinvarianten zur
Weylgruppe}, J. Amer. Math. Soc., \textbf{3} 1990, 421--445.

\bibitem{Sto}
{\sc A.~Stokke}, {\em A symplectic D\'esarm\'enien matrix and a standard basis for the symplectic Weyl module},
 J. Alg., \textbf{272} (2004), 512--529.

\bibitem{Stro1} {\sc C.~Stroppel},
{\em Category $\mathcal{O}$: quivers and endomorphism rings of projectives}, Represent. Theory, \textbf{7} (2003), 322--345.


\bibitem{Ta}
{\sc H.~Tachikawa}, {\em Quasi-Frobenius Rings and Generalizations}, Springer Lecture Notes in
Mathematics, Vol. \textbf{351}, Springer-Verlag, Berlin/New York, 1973.

\bibitem{Wenzl} {\sc H.~Wenzl}, {\em On the structure of Brauer's centralizer algebras},
Ann. of Math., \textbf{128} (1988), 173--193.

\bibitem{Weyl} {\sc H.~Weyl}, {\em The classical groups, their invariants and representations},
Princeton University Press, 1940.

\end{thebibliography}
\end{document}